\documentclass[a4paper]{article}

\usepackage[latin1]{inputenc}
\usepackage{bbm}
\usepackage{amsmath,amsthm,amssymb}
\usepackage{graphicx,color}

\newcommand{\com}[1]{}
\newcommand{\R}{\mathbb{R}}
\newcommand{\N}{\mathbb{N}}
\newcommand{\Z}{\mathbb{Z}}
\newcommand{\E}{\mathbb{E}}
\newcommand{\Prob}{\mathbb{P}}
\newcommand{\Probb}{\mathbf{P}}
\newcommand{\ProbX}{\mathbb{P}_X}

\newcommand{\Ss}{\mathcal{S}}
\newcommand{\U}{\mathcal{U}}
\newcommand{\Unif}{\mathcal{U}\bigl([0,1]\bigr)}
\newcommand{\Unifd}{\mathcal{U}\bigl([0,1]^d\bigr)}
\newcommand{\opSuppP}{\mathring{\overbrace{\supp(\Probb)}}}
\newcommand{\PSpace}{(\Omega_0 \times \Omega,  {\Ss}_0\otimes \Ss,
\Prob_0\otimes\Prob)}
\newcommand{\Qpn}{Q^{\text{dq}}_{\norm{\cdot}, p, d}}
\newcommand{\VQpn}{Q^{\text{vq}}_{\norm{\cdot}, p, d}}
\newcommand{\limn}{\lim_{n\to\infty}}
\newcommand{\limsn}{\limsup_{n\to\infty}}
\newcommand{\limin}{\liminf_{n\to\infty}}
\newcommand{\Cie}{C_{i,\varepsilon}}
\newcommand{\JG}{\mathcal{J}_\Gamma}

\newcommand{\Fp}{F^p}

\newcommand{\Fbp}{\bar F^p}

\newcommand{\Gae}{\Gamma_{\!\!\alpha, \varepsilon}}
\newcommand{\gaei}{\Gamma_{\alpha_n, \varepsilon, i}}
\newcommand{\Gammai}{\Gamma_{\!i}}
\newcommand{\Gamman}{\Gamma_{\!n}}
\newcommand{\Gammaone}{\Gamma_{\!1}}
\newcommand{\Gammatwo}{\Gamma_{\!2}}
\newcommand{\Gammathree}{\Gamma_{\!3}}

\newcommand{\dqp}{d^p}
\newcommand{\dqpn}{\dqp_{n,p}}
\newcommand{\dqbp}{\bar d^p}
\newcommand{\dqbpn}{\dqbp_{n,p}}

\providecommand{\dqpnn}[1]{\dqp_{#1}}

\newcommand{\Dqp}{d_p}
\newcommand{\Dqpn}{d_{n,p}}

\newcommand{\Dqbpn}{\bar d_{n,p}}

\providecommand{\abs}[1]{\lvert#1\rvert}

\providecommand{\biggabs}[1]{\biggl\lvert#1\biggr\rvert}
\providecommand{\norm}[1]{\lVert#1\rVert}

\providecommand{\normdp}[1]{\norm{#1}_{d/(d+p)}}
\providecommand{\ind}[1]{\mathbbm{1}_{#1}}
\providecommand{\consLP}[1]{\text{s.t. } \left[  \begin{smallmatrix}
          x_1 & \ldots & x_n\\
          1 & \ldots & 1\\
        \end{smallmatrix}  \right] \lambda =
      \left[\begin{smallmatrix}
         #1\\ 1\\
        \end{smallmatrix} \right],\,
      \lambda \geq 0 }

\providecommand{\LP}[1]{
\underset{\consLP{#1}}{\min_{\lambda\in\R^n}\sum_{i=1}^n \lambda_i \, \norm{#1-x_i}^p}}

\def \Prod{\displaystyle\prod}

\DeclareMathOperator{\supp}{supp}

\DeclareMathOperator{\conv}{conv}

\DeclareMathOperator{\dist}{dist}

\DeclareMathOperator{\adim}{aff.dim}

\theoremstyle{plain}
\newtheorem{prop}{Proposition}

\newtheorem{lemma}{Lemma}
\newtheorem{thm}{Theorem}

\theoremstyle{remark}
\newtheorem*{remark}{Remark}

\title{Sharp rate for the dual quantization problem}

\author{Gilles Pag\`es\thanks{Laboratoire de Probabilit\'es et Mod\`eles al\'eatoires, UMR~7599, Universit\'e Paris 6, case 188, 4,
pl. Jussieu, F-75252 Paris Cedex 05. E-mail: {\tt  gilles.pages@upmc.fr}}
 \and Benedikt Wilbertz\thanks{Laboratoire de Probabilit\'es et Mod\`eles al\'eatoires, UMR~7599, Universit\'e Paris 6, case 188, 4,
pl. Jussieu, F-75252 Paris Cedex 05. E-mail: {\tt  benedikt.wilbertz@gmx.de}}
 }

\begin{document}
\maketitle

\begin{abstract}
In this paper we establish the sharp rate of the optimal dual quantization problem.
The notion of dual quantization was recently introduced in 
\cite{dualStat}, where it has been shown that, at least in a Euclidean setting, dual
quantizers are based on a Delaunay triangulation, the dual counterpart of the
Voronoi tessellation on which ``regular'' quantization relies.
Moreover, this new approach shares an intrinsic stationarity property, which
makes it very valuable for numerical applications.

We establish in this paper the counterpart for dual quantization of the
celebrated Zador theorem, which describes the sharp asymptotics for the
quantization error when the quantizer size tends to infinity.  On the way we establish an extension of the so-called Pierce Lemma by a random quantization argument. Numerical results confirm our choices. 

%
\end{abstract}

\bigskip
\noindent {\em Keywords: quantization, quantization rate, Zador's Theorem, Pierce's Lemma,
dual quantization, Delaunay triangulation, random quantization.}

\section{Introduction}
Starting with~\cite{cdo} and continued in~\cite{dualStat}, we introduced a new notion of vector quantization called {\em
dual quantization} (or  {\em Delaunay quantization}  in a Euclidean framework). We developed in~\cite{dualAppl} some first applications towards the design of numerical schemes for multi-dimensional optimal stopping and stochastic control problems arising in Finance (see also~\cite{BCC}).
In general, the principle of dual quantization consists of mapping an $\R^d$-valued random vector 
(r.v.) onto a non-empty finite subset (or {\em grid}) $\Gamma\subset \R^d$ using an appropriate random
splitting operator ${\cal J}_{\Gamma} : \Omega_0\times \R^d\to \Gamma$ (defined
on an exogenous probability space $(\Omega_0,\mathcal{S}_0, \Prob_0)$)
which satisfies the intrinsic {\em stationarity property}
\begin{equation}\label{1a}
\forall\, \xi \!\in \conv(\Gamma),\qquad \E_{\Prob_0}({\cal J}_{\Gamma}
(\xi))=\int_{\Omega_0} {\cal J}_{\Gamma} (\omega_0,\xi)
\,\Prob_0(d\omega_0)=\xi,
\end{equation}
where $\conv(\Gamma)$ denotes the convex hull of $\Gamma$ in $\R^d$. Every r.v.  $X:(\Omega,\mathcal{S}, \Prob)\to \conv(\Gamma) $ defined on a probability space   can be  canonically
extended to $(\Omega_0\times\Omega,\mathcal{S}_0\otimes\mathcal{S},
\Prob_0\otimes\Prob)$ in order to define  {\em dual quantization} induced by $\Gamma$ as
\[
\widehat X^{\Gamma, dual}(\omega_0, \omega) =   {\cal J}_{\Gamma} (\omega_0,X(\omega)).
\]
As a specific feature inherited from~\eqref{1a}, it always satisfies the {\em dual} or {\em reverse stationary property}
 \[ 
 \E_{\Prob\otimes\Prob_0}({\cal J}_{\Gamma} (X)\,|\,
X)=X. 
\] 
This 
can be compared to the more  classical Voronoi framework   where the $\Gamma$-quantization of $X$ is defined from a  Borel nearest neighbour projection ${\rm Proj}_{\Gamma}$ by
\[
\widehat X^{\Gamma, vor}( \omega) = {\rm Proj}_{\Gamma}(X(\omega)).
\]
The stationary property then reads: $\displaystyle \E(X\,|\, \widehat X^{\Gamma, vor})=\widehat X^{\Gamma, vor}$, except that it holds 
only for grids which are critical points (typically local minima) of the so-called distortion function (see $e.g.$~\cite{Foundations})
in a Euclidean framework.

To each  quantization is corresponds  a functional approximation operator:    Voronoi quantization is related to the {\em stepwise constant functional  approximation operator}~$f \!\circ {\rm Proj}_{\Gamma}$ whereas dual quantization leads  to an operator defined for every $\xi\!\in\conv(\Gamma)$ by
\begin{equation}\label{eq:JGAmma}
\mathbb{J}_{\Gamma}(f)(\xi)= \E_{\Prob_0}\big(f(J_{\Gamma}(\omega_0,\xi))\big)= \sum_{x\in \Gamma} f(x)\lambda_x(\xi),
\end{equation} 
where $\lambda_x(\xi)=\Prob_0(J_{\Gamma}(.,\xi)=x)$, $x\!\in \Gamma$, are barycentric ``pseudo-coordinates" of $\xi$ in $\Gamma$ satisfying $\lambda_x(\xi)\!\in [0,1]$, $\sum_{x\in \Gamma} \lambda_x(\xi) =1$ and $\sum_{x\in \Gamma} \lambda_x(\xi) x=\xi$. The operator $\mathbb{J}_{\Gamma}$  is an {\em interpolation} operator which turns out,  under appropriate conditions, to be  more regular (continuous and stepwise affine, see~\cite{dualAppl}) than the ``Voronoi" one.
 It is shown  in~\cite{dualStat,cdo,dualAppl} how we can take advantage of this intrinsic stationary property to  produce more accurate error bounds for the resulting  {\em cubature formula}
 \begin{equation}\label{eq:cubature}
\E_{\Prob}(f(\widetilde X^{\Gamma,dual}))=  \E_{\Prob}(\mathbb{J}_{\Gamma}(f)(X))=   \E_{\Prob\otimes\Prob_0} \big(f(J_{\Gamma}(\omega_0,\xi))\big)= \sum_{x\in \Gamma}w_x^{dual}f(x) 
 \end{equation}
where $w_x^{dual} = \E_{\Prob}(\lambda_x(X))= \Prob\otimes\Prob_0(J_{\Gamma}(\omega_0,X)=x)$, $x\!\in \Gamma$, regardless of any optimality property  $\Gamma$ with respect to   $\Prob_{X}$.   Typically, if  $f\! \in {\rm Lip}(\R^d,\R)$ (Lipschitz continuous function)  with  coefficient~$[f]_{\rm Lip}$, 
\begin{eqnarray*}
  \Big|\E_{\Prob}f(X)-\E_{\Prob\otimes\Prob_0}f(\widetilde X^{\Gamma,dual})\Big|&\le& [f]_{\rm Lip}\big\|X-\widehat X^{\Gamma,dual}\big\|_{L^1(\Prob\otimes\Prob_0)}\\&=&  [f]_{\rm Lip}\E_{\Prob \otimes \Prob_0}\big(\|X-J_{\Gamma}(\omega_0,X)\| \big)   \\
  &=& [f]_{\rm Lip}\E_{\Prob \otimes \Prob_0}\big(\E_{\Prob \otimes \Prob_0}(\|X-J_{\Gamma}(\omega_0,X)\|  \,|\,X)  \big)   
\end{eqnarray*}
whereas, if $f$ has Lipschitz continuous differential (the norm on $\R^d$ is denoted $\|\,.\,\|$), a second order Taylor expansion yields
\begin{eqnarray}
\nonumber    \Big|\E_{\Prob}f(X)-\E_{\Prob\otimes\Prob_0}f(\widetilde X^{\Gamma,dual})\Big|&\le& \Big\|f(X) -\E_{\Prob \otimes \Prob_0}\big(f(J_{\Gamma}(\omega_0,X))\,|\,X\big) \Big\|_{L^1(\Prob\otimes\Prob_0)}    \\
\nonumber     &\le& [ Df]_{\rm Lip}\E_{\Prob \otimes \Prob_0}\big(\|X-J_{\Gamma}(\omega_0,X)\|^2  \big)  \\
  &\le& [ Df]_{\rm Lip}\E_{\Prob \otimes \Prob_0}\big(\E_{\Prob \otimes \Prob_0}(\|X-J_{\Gamma}(\omega_0,X)\|^2  \,|\,X)  \big)   \label{eq:fstatiodual}
\end{eqnarray}
where $\displaystyle \E_{\Prob\otimes\Prob_0}(\|X-J_{\Gamma}(\omega_0,X)\|^p\,|\,X)= \sum_{x\in\Gamma}\lambda_x(X)\|X-x\|^p= \mathbb{J}_{\Gamma}(\|.\|^2)(X),\; p=1,2$.

More generally, if one aims at approximating $\E\big (f(X)\,|\, g(Y) \big)$ by its dually quantized counterpart $\E_{\Prob\otimes\Prob_0\otimes\Prob_1}\big (f(J_{\Gamma_X}(\omega_0,X))\,|\, J_{\Gamma_Y}(\omega_1,Y) \big)$  (with obvious notations),  it is also possible under natural additional assumptions  to get  error bounds based on both related dual quantization error moduli, see e.g. the proof (Step~2) of ~Proposition~2.1 in~\cite{dualAppl}.

This suggests  to investigate the properties and the asymptotic    behaviour  of the $(\Gamma, L^p)$-mean dual quantization error, $p\!\in (0,\infty)$, defined by
$$
\Big\|X-\widehat X^{\Gamma,dual}\Big\|^p_{L^p(\Prob\otimes\Prob_0)}= \Big\|X-J_{\Gamma}(\omega_0,X)\Big\|^p_{L^p(\Prob\otimes\Prob_0)}= \E_{\Prob\otimes\Prob_0} \Big( \E_{\Prob\otimes\Prob_0}\big(\|X-J_{\Gamma}(\omega_0,X)\|^p\,|\, X\big) \Big)
$$ 
so as to make it as small as possible. This program can be summed up in four phases:

\smallskip
-- The first step is to minimize the above conditional expectation, $i.e.$  $\E(\|\xi-J_{\Gamma}(\omega_0,\xi)\|^p)$ for every $\xi \!\in\conv(\Gamma)$, for a fixed grid $\Gamma$ $i.e.$ to determine the {\em best}  splitting random operator $J_{\Gamma}$. In a  regular quantization, this phase corresponds to showing that the nearest neighbour projection on $\Gamma$ is   the best projection on $\Gamma$.

\smallskip
-- The second step is ``optional" . It aims at finding grids which minimize the
mean dual quantization error $\Big\|X-J_{\Gamma}(\omega_0,X)\Big\|_{L^p(\Prob\otimes\Prob_0)}$ {\em among all grids $\Gamma$ whose convex hull contains the support of the distribution of $X$} or equivalently such that $\Prob(X\!\in \conv(\Gamma))=1$.

\smallskip
-- The third step is to extend  dual quantization to r.v.s $X$ with unbounded support while  the performances of the resulting  cubature formula (see~\eqref{eq:fstatiodual}), having in mind that  the stationarity can no longer holds.

\smallskip
The first two steps have  been already solved in~\cite{dualStat}. We discuss  in-depth the third one in Section~\ref{sec:motiv}). The aim of this paper is to solve the fourth and last  step:   elucidate  is the rate of  decay  to $0$ of the   {\em optimal} $L^p$-mean dual quantization error modulus, $i.e.$ minimized over all grids $\Gamma$ of size at most $N$ -- as $N$ grows to infinity.

This is the  to establish  in a dual
quantization framework the counterpart of   Zador's Theorem  which
rules the convergence rate of  optimal ``regular" (Voronoi) quantization and is
recalled below. 
To be more precise, we will  establish such a  theorem, for $L^{\infty}$-bounded r.v.s but also, 
once mean dual quantization error will have been extended in an appropriate way
following~\cite{dualStat}, to general r.v.s.

\smallskip
Let us now introduce in more formal way   the (local and mean) dual quantization error moduli,  following~\cite{dualStat}.  For a grid $\Gamma\subset\R^d$, we define  the {\em $L^p$-mean dual quantization  error} of  $X$ induced by the grid $\Gamma $ by
\begin{equation}\label{def:dpGamma}
d_{p}(X; \Gamma)= \|F_p(X;\Gamma)\|_{L^p(\Prob)}
\end{equation}
where $F_p$ denotes the {\em  local dual quantization error} function  defined by
\begin{eqnarray}
\label{eq:Fp} F_p(\xi;\Gamma)&=&\inf\left \{  \Big(\sum_{x\in \Gamma} \lambda_x \|\xi-x\|^p
\Big)^{\frac 1p},\; \lambda_x\!\in [0,1],\, \sum_{x\in \Gamma}
\lambda_x\, x =\xi,\,
\sum_{x\in \Gamma} \lambda_x=1\right \} 
\end{eqnarray}
Note that $F_p(\xi;\Gamma)<+\infty$ if and only if $\xi\!\in \conv(\Gamma)$ so that  $d_{p}(X; \Gamma)<+\infty $ if and only if $ X\!\in \conv(\Gamma)$ $\Prob$-$a.s.$. and that $d_{p}(X; \Gamma)= \Big\|X-\widehat X^{\Gamma,dual}\Big\|^p_{L^p(\Prob\otimes\Prob_0)}$. Hence, this notion only makes sense for compactly supported r.v.s. In particular if the support of $\Prob_{_X}$ is compact and contains $d+1$ affinely independent points, $d_{n,p}(X,\Gamma)=+\infty$ as long as $n\le d$.This new quantization modulus leads to an optimal dual quantization problem at {\em level} $N$, 
\begin{equation}\label{def:dnpsharp}
d_{n,p}(X)\!=\!\inf\Big\{d_{n,p}(X, \Gamma),\, \Gamma\!\subset \R^d, \; |\Gamma|\le
n\Big\}
\!=\!\inf\Big\{\|F_p(X;\Gamma)\|_p,\, \Gamma\!\subset \R^d, \; |\Gamma|\le
n\Big\}.
\end{equation}

One important application of quantization in general is the use of quantization grids as numerical cubature formula (see~\eqref{eq:cubature}). 
The main feature here is the stationarity which allows to derive a second order formula for the integration error. 
Since, by construction, dual quantization can  achieve stationarity only on a compact set, we show in section 2.2 that the extension of dual quantization to non-compactly supported random variables as defined in~\cite{dualStat} preserves this second order rate on the whole support of the r.v.

We therefore define the splitting operator $\mathcal{J} _{\Gamma}$ outside $\conv(\Gamma)$ by setting

\[ 
\forall\,\xi \in \R^d\setminus
\conv(\Gamma),\quad \mathcal{J}_{\Gamma}(\omega_0,\xi)= {\rm
Proj}_{\conv{(\Gamma)}\cap \partial\Gamma}(\xi) 
\] 
where $ {\rm Proj}_{\conv{(\Gamma)}\cap \partial\Gamma}$ is a Borel nearest neighbour projection on $\conv{(\Gamma)}\cap \partial\Gamma$. This choice is not unique: an alternative extension could be  to set   $\mathcal{J}_{\Gamma}(\omega_0,\xi)={\rm Proj}_{\conv(\Gamma)}(\xi)$. 
But the above  choice is   tractable in terms of simulation and we will prove that it does not deteriorate  the resulting mean error when $|\Gamma|\to +\infty$. Though the stationary property is lost as expected, we point out  in Section~\ref{sec:motiv} that this operator remains as performing as $\mathcal{J}_{\Gamma}$ is for bounded r.v.s when implementing cubature formulas for unbounded r.v.s. 

%

Then, we to derive the {\em extended  local dual quantization error} function by
\begin{equation}\label{def:Fbarp}
\bar F_p(\xi;\Gamma):= F_p(\xi;\Gamma)\,\mbox{\bf 1}_{\conv(\Gamma)}(\xi)+{\rm
dist}(X,\Gamma)\,\mbox{\bf 1}_{\conv(\Gamma)^c}(\xi),
\end{equation}
and  the {\em  extended $L^p$-mean dual quantization error} of $X$
induced by  $\Gamma$ by  
\begin{equation}\label{def:dbarpGamma}
\bar d_{p}(X; \Gamma)= \|\bar F_p(X;\Gamma)\|_{L^p(\Prob)}.
\end{equation}
Finally, we define   the {\em extended $L^p$-mean dual quantization error}  at level $n$ given by
\begin{equation}\label{def:dbarnp}
\bar d_{n,p}(X)=\inf \Big\{\bar d_p(X,\Gamma),\; \Gamma\!\subset \R^d, \;
|\Gamma|\le n\Big\}.
\end{equation}


Finally, we briefly recall a few facts about  the (regular) {\em Voronoi optimal quantization problem} at level $n$ associated to the nearest neighbour projection ${\rm Proj}_{\Gamma}$: it reads
\begin{equation}\label{def:enp}
e_{n,p}(X)=\inf \left\{ \| {\rm dist}(X,\Gamma)\|_{L^p(\Prob)} ,\; \Gamma\subset \R^d,\, |\Gamma|\le n\right\}
\end{equation}
(where ${\rm dist}(x,A)= \inf_{a\in A}\|x-a\|$). It is well-known that $e_{n,p}(X)\downarrow 0$ as soon as $n\to +\infty$ and
$X\!\in L^p(\Prob)$. Moreover, the rate of convergence to $0$ of
$e_{n,p}(X)$ is ruled by  Zador's  Theorem (see~\cite{Foundations}).

\begin{thm}[Zador]
Let $X\!\in L_{\R^d}^{p'}(\Prob)$, $p'\!>\!p$. Let $\ProbX= h.\lambda_d +\nu$, $\nu\perp\lambda_d$ be the distribution of $X$ where $\lambda_d$ denotes the Lebesgue measure on $(\R^d,{\cal B}or(\R^d))$.
Then
\[
\limn n^{\frac 1d} e_{n,p}(X) =
\VQpn\, \|h\|_{\frac{d}{p+d}}^{\frac 1p}
\] 
where $ \|h\|_{\frac{d}{p+d}}=\displaystyle \Big(\int_{\R^d} h(\xi)^{\frac{d}{p+d}}d\xi\Big)^{1+\frac pd}$ and $\displaystyle \VQpn= \inf_{n} n^{\frac 1d}\,e_{n,p}(U([0,1]^d))\!\in
(0,\infty)$.
\end{thm}

The above  rate depends on $d$ and  is known as the {\em curse of dimensionality}.
 Its statement  and proof goes back to Zador (PhD, 1954) for the uniform distributions on hypercubes, its 
extension to   absolutely  continuous distributions is due to 
Bucklew and Wise in~\cite{bucklew}. A first general rigor proof 
 (according to mathematical  standards) was provided 
in~\cite{Foundations} in 2000 (see also \cite{GRAY} for a survey of the history of quantization).

\smallskip
It should be noted that $d_{n,p}(X)$ and $\bar d_{n,p}(X)$ do not coincide even for bounded r.v.s. We will extensively use (see~\cite{dualStat}) that
\[
 d_{n,p}(X)\ge  \bar d_{n,p}(X)\ge e_{n,p}(X).
\]

This paper is entirely devoted to establishing the sharp asymptotics of the
optimal dual quantization error moduli $d_{n,p}(X)$ and $\bar d_{n,p}(X)$ as $n$ goes to infinity. 
The main result is stated in Theorem~\ref{thm:DQRate} (Zador's like theorem)
(see Section~\ref{mains} below). 
Proposition~\ref{PdtQErrop} (a Pierce like Lemma) is a companion result
which provides a non-asymptotic upper bound for the exact rate simply involving  
moments of the r.v. $X$ (higher than $p$).  Our   proof   has the same structure  as that of the original Zador  Theorem (see~$e.g.$~\cite{Foundations} where it has been rigorously completed for   the first time), except that the  splitting operator $\mathbb{J}_{\Gamma}$ is much more demanding to handle than the plain nearest neighbour projection: it requires more sophisticate arguments borrowed from convex analysis (including dual primal/methods) and geometry,  both in  a probabilistic framework.  In one dimension the exact rate $O(n^{-1})$ for $d_{n,p}(X)$ and $\bar d_{n,p}(X)$ follows from a random quantization argument detailed in Section~\ref{Pierce} (extended Pierce Lemma for $d_{n,p}(X)$).
This rate can be transferred in a $d$-dimensional framework  to  $O(n^{-\frac
1d})$  using  a product (dual) quantization
argument (see~Proposition~\ref{prop:asymptQ} below and Section~\ref{sec:localProperties}). Finally, the sharp upper bound
is obtained in Section~\ref{sec:rate} by successive approximation procedures of
the density of $X$,  whereas the lower bound relies on a new ``firewall" Lemma.

\smallskip
\noindent {\sc Notations}:  $\bullet$ $\conv(A)$ stands for the convex hull of $A\subset \R^d$,
$\abs{A}$ for its cardinality, ${\rm diam}_{\norm{.}}(A)=\sup_{x,y\in A}\norm{x-y}$ for its diameter and ${\rm aff.dim}(A)$ for the dimension of the affine subspace of $\R^d$ spanned by $A$. 

\noindent  $\bullet$ We denote 
$ \big(\begin{smallmatrix} n\\ i \end{smallmatrix}\big):=\frac{n!}{i!(n-i)!}, \; n,\, i\!\in \{0,\ldots,n\}$, $n\!\in \mathbb N$.
 
\noindent $\bullet$ $\lfloor x \rfloor$ and $\lceil x\rceil$ will denote the lower and the upper
integral part of the real number $x$ respectively; set likewise
$x_{\pm}=\max(\pm x,0)$.  For two
sequences of real numbers $(a_n)$ and $(b_n)$,  $a_n\sim b_n$ if $a_n =u_n b_n$ with $\lim_n u_n =1$.

\noindent $\bullet$  For every $x=(x^1,\ldots,x^d)\!\in \R^d$, $|x|_{\ell^r}=(|x^1|^r
+\cdots|x^d|^r)^{1/r}$ denotes the $\ell^r$-norm or pseudo-norm, $0<r<+\infty$  and $|x|_{\ell^{\infty}}=\max_{1\le i\le d}|x_i|$ denotes the $\ell^{\infty}$-norm. A general norm on $\R^d$ will be denoted $\norm{\cdot}$.
 
\noindent $\bullet$ ${\rm supp}(\mu)$ denotes the support of a distribution $\mu$ on $(\R^d,{\cal B}or(\R^d))$. 

\section{Main results and motivation for extended dual quantization}
\subsection{Main results}\label{mains} 
The theorem below establishes  for any $p>0$ and any norm on $\R^d$
the counterpart of Zador's Theorem in the framework of dual quantization for both
$d_{n,p}$ and $\bar d_{n,p}$   error moduli.

\begin{thm}\label{thm:DQRate} $(a)$ 
Let $X\!\in L_{\R^d}^{\infty}(\Prob)$. Assume the distribution $\ProbX$ of
$X$ reads $\ProbX= h.\lambda_d+\nu$, $\nu\perp\lambda_d$. Then
\[
\limn n^{\frac 1d}\, d_{n,p}(X) 
= \limn n^{\frac 1d}\, \bar d_{n,p}(X)
= \Qpn\,\|h\|_{\frac{d}{p+d}}^{\frac 1p}
\] 
where $\displaystyle \Qpn= \inf_{n\ge 1} n^{\frac 1d}\, d_{n,p}(U([0,1]^d))\!\in
(0,\infty)$.

\noindent $(b)$  Let $X\!\in L_{\R^d}^{p'}(\Prob)$, $p'>p$. Assume the
distribution $\ProbX$ of $X$ reads $\ProbX= h.\lambda_d+\nu$,
$\nu\perp\lambda_d$. Then
\[
\limn n^{\frac 1d}\, \bar d_{n,p}(X) 
= \Qpn\,\|h\|_{\frac{d}{p+d}}^{\frac 1p}.
\] 
\noindent $(c)$ If $d=1$, then 
\[
d_{n,p}(U([0,1])) = \biggl(\frac{2}{(p+1)(p+2)} \biggr)^{\frac 1p}
\frac{1}{n-1},
\]  
which implies $Q^{dq}_{|\,.\,|,p,1}=\left(\frac{2^{p+1}}{p+2}\right)^{\frac
1p}Q^{vq}_{|\,.\,|,p,1}$.
\end{thm}

Moreover, we will also establish in Section~\ref{sec:rate} an upper bound for the
dual quantization coefficient $\Qpn$ when $\norm{\cdot} =
\abs{\cdot}_{\ell^r}$.
\begin{prop}[Product quantization]\label{prop:asymptQ}
Let $r,p\in[1,\infty)$ with $r\leq p$.
Then it holds for every $d\in\N$
\[
	Q^{\text{dq}}_{\abs{\cdot}_{\ell^r}, p, d} \leq d^{\frac{1}{r}}\cdot
	Q^{\text{dq}}_{\abs{\cdot}, p, 1}
\]
where $|\,.\,|$ denotes standard absolute value on $\R$.\end{prop}

Since this upper bound achieves the same asymptotic rate as in the case of
regular quantization (cf. Corollary~9.4 in~\cite{Foundations}), this suggests the rate
 $O(d^\frac{1}{r})$ to be also the true one for $\Qpn$ as $d\to\infty$.

As a step towards the above sharp rate theorem, we  also establish a counterpart
of the so-called Pierce Lemma (as stated in an operating form $e.g.$ 
in~\cite{meanRegular}). 
In practice, it turns out to be quite  useful for applications since it provides 
non-asymptotic error bounds which  only depend on the moments of the r.v. $X$
and the size of the optimal grid as emphasized in~\cite{dualAppl} (see section
\ref{sec:ddimBound} for the proof).

 \begin{prop}[$d$-dimensional extended Pierce Lemma] \label{PdtQErrop} $(a)$ Let
 $p,\,\eta>0$. There exists 
 a real constant $C_{d, p,\eta}>0$ such that, for every $n\ge 1$
and every r.v.  $X\!\in L_{\R^d}^{p+\eta}(\Omega,{\cal A}, \Prob)$, 
 \[ 
 \bar d_{n,p}(X)\le C_{d,p,\eta}\sigma_{p+\eta,\|.\|}(X)\,n^{-1/d}
 \] 
 where $\sigma_{p+\eta,\|.\|}(X)= \inf_{a\in \R^d} \|X-a\|_{L^{p+\eta}}$ denotes the $L^{p+\eta}$-pseudo-standard deviation of $X$. 
 
 \smallskip
\noindent $(b)$  If $\supp(\ProbX)$ is compact then there exists 
 a real constant $C'_{d, p,\eta}>0$ such that, for every $n\ge 1$
 \[
 d_{n,p}(X)\le C'_{d,p,\eta}{\rm diam}_{\norm{.}}({\rm supp}(\Prob_{_X}))\,n^{-1/d}. 
 \]
\end{prop}


\subsection{How to use the extended $L^p$-dual quantization error modulus?} \label{sec:motiv}
%

We  briefly explain why the extended dual quantization error modulus,  already been  introduced in~\cite{dualStat} for non-compactly supported distributions, is the right  tool to {\em perform automatically an optimized truncation} of non-compactly supported distributions.  basically, it uses its  additional  ``outer Voronoi projection"   as a {\em penalization term}   which expands automatically the convex hull of the dually optimal   grid at its appropriate ``amplitude", making altogether   the distribution outside of its  convex hull  ``negligible"   and sharing  an optimal  rate of decay $n^{-\frac  1d}$ as its size $n$ goes to infinity. The specific choice of a Voronoi quantization among other possible  solutions  for this penalization is motivated by both  its   theoretical tractability and its simple implementability in stochastic grid optimization algorithms.
This feature if of the highest importance for numerical integration or conditional execration approximation. This is the main motivation  to introduce and deeply  investigate  the sharp asymptotics of  this   $L^p$-mean extended dual quantization error modulus  $\bar d_{n,p}(X)$. 

We saw in~\cite{dualStat} that {\em Euclidean  dual quantization}  of a compactly supported distribution produces {\em  stationary} (dual) quantizers, namely r.v.s $\widehat X^{dual}$ satisfying $\E(\widehat X^{dual}\,|\, X)=X$, so that (see Proposition~?? in~\cite{dualStat}), dual quantization based cubature formula induce on functions $f\!\in {\cal C}_{\rm Lip}^1(\R^d,\R)$ (Lipschitz functions with Lipschitz continuous gradient) an error at most equal to $[D f]_{\rm Lip}d_{2,n}(X)^2$. Taking into account the rate established in Theorem~\ref{thm:DQRate}$(a)$, this yields  a $O(n^{-\frac{2}{d}})$ error rate.

There is no way to extend dual quantization to (possibly) unbounded r.v.s so that it preserves the above stationarity property. However, with the choice we made (nearest neighbor projection on the  grid outside its convex hull), natural heuristic arguments strongly suggest that the above order $O(n^{-\frac{2}{d}})$ is still satisfied for functions in  $ {\cal C}_{\rm Lip}^1(\R^d,\R)$.

We consider an  unbounded Borel  distribution $\mu=\Probb_{_X}$ of an $\R^d$-valued r.v. $X$. Let $\Gamma_{n}$ be an {\em  Euclidean $L^2$-optimal  extended} dual quantization grid of size $n$ for $\mu$ (see~\cite{dualStat} or  Theorem~\ref{thm:existence}) and $\widehat X^{dual}$
the resulting $\Gamma_n$-valued extended dual quantization of $X$. Let $C_{n}= \conv\big(\Gamma_{n}\big) $ denote the  convex hull of $\Gamma_n$. It is clear by construction of $\widehat X^{dual}$  that $\widehat X^{dual}\!=\! \widetilde X^{dual}
+\!\widetilde X^{vor}$
where, with obvious notations, 
\[
\mbox{\bf 1}_{\{X\in C_{n}\}} \E\big(\widetilde X^{dual}|X\big) = \mbox{\bf 1}_{\{X\!\in C_{n}\}} X\; \mbox{ (dual stationarity) and }\;  \widetilde X^{vor}= {\rm Proj}_{\Gamma_n \cap C_n}(X).
\]
Hence, if $f\!\in {\cal C}_{\rm Lip}^1(\R^d,\R)$, $ \E\big((D f(X)|X-\widetilde X^{dual}) |X\!\in C_{n}\big)=0$ and 
\begin{eqnarray*}
\left|\E\Big(f\big(\widetilde X^{dual}\big) |X\!\in C_{n}  \Big)\!-\!\E\Big(f\big( X\big)|X\!\in C_{n}\Big)\right|&\!=\!&\left| \E\left(f\big(\widetilde X^{dual}\big)\!-\!f\big( X\big) \!-\!D f(X).(X\!-\!\widetilde X^{dual}) |X\!\in C_{n}\right)\right|\\
&\!\le\! & [D f]_{\rm Lip}d_{n,2}(\Gamma_{n},\widetilde X^{dual}| X\!\in C_{n})^2.
\end{eqnarray*}
Consequently, 
\begin{eqnarray*}
\left|\E\Big(f\big(\widetilde  X^{dual}\big)\mbox{\bf 1}_{\{X\in C_{n}\}} \Big) -\E\Big(f\big( X\big)\mbox{\bf 1}_{\{X\in C_{n}\}} \Big)\right|&\le&  [D f]_{\rm Lip}d_{n,2}(\widetilde X^{dual},\Gamma_{n})^2/\Probb( X\!\in C_{n})\\ &\le&  [D f]_{\rm Lip}\bar d_{n,2}(X,\Gamma_{n})^2/\Probb( X\!\in C_{n}).
\end{eqnarray*}
On the other  hand, 
\begin{eqnarray*}
\left|\E\Big(f\big(\widetilde X^{vor}\big) \mbox{\bf 1}_{\{X\notin C_{n}\}} \Big)-\E\Big(f\big( X\big) \mbox{\bf 1}_{\{X\notin C_{n}\}} \Big)\right|&\le& [f]_{\rm Lip}  e_{n,2}\big(X,\Gamma_n\big)\Probb\big(X\notin C_{n}\big)^{\frac 12}\\
& \le& [f]_{\rm Lip}\bar d_{n,2}(X)\Probb\big(X\notin C_{n}\big)^{\frac 12}. 
\end{eqnarray*}
Relying on Theorem~\ref{thm:DQRate}$(b)$, we know that, if $\mu = h.\lambda_d \stackrel{\perp}{+} \nu$, then  $\bar d_{n,2}(X)\sim  Q^{dq}_{2,|.|_{eucl}}\,\|h\|_{\frac{d}{2+d}}^{\frac 1p} n^{-\frac 1d}$. 
The ``outside" contribution will be negligible compared to the    ``inside" one as soon as 
\begin{equation}\label{eq:PCn}
\Probb\big(X\notin C_{n}\big) = o\Big( \bar d_{n,2}(X,\Gamma_{n})^2 \Big)=  o\Big( n^{-\frac 2d}\Big).
\end{equation}

 This condition turns out to be  not very demanding and can be checked, at least heuristically, as illustrated below: if $X\stackrel{d}{=} {\cal N}(0;I_d)$, one may conjecture, taking advantage of the spherical symmetries of the normal  distribution, that $C_{n}$ is approximately a sphere centered at $0$ with radius $\rho_{n}= \max_{a\in \Gamma_{n}}|a|$.
As 
 \[
 \Probb(|X|\ge \xi) \sim V_d\,\xi^{d-2}e^{-\frac{\xi^2}{2}}\;\mbox{ as }\; \xi\to +\infty\quad (\mbox{with }V_d=\lambda_{d-1}(S_d(0,1))).
 \]
Condition~\eqref{eq:PCn} is satisfied as soon as $\liminf_n \frac{\rho_{n}}{\sqrt{\log n}}>\frac{2}{\sqrt{d}}$ ($\ge $ if $d=1,2$). As an example, one must have in mind that, for optimal {\em Voronoi} quantization, this inequality is satisfied since (see~\cite{PASA}) $\lim_n \frac{\rho_{n}}{\sqrt{\log n}}=\sqrt{2(1+2/d)}>\frac{2}{\sqrt{d}}$. More precisely, we have
\[
\Probb\big(X\notin C_{n}\big)\sim \kappa_d(\log n)^{\frac d2-1}n^{-1-\frac 2d} \; \mbox{ so that }\; \bar d_{n,2}(X)\Probb\big(X\notin C_{n}\big)^{\frac 12} = O\big(n^{-\frac 2d-\frac 12}(\log n)^{\frac{d-2}{4}}\big).
\] 
 
Numerical experiments, not reproduced here, carried out with the above ${\cal N}(0;I_d)$ distribution confirm   that the radius of optimal dual quantizers always achieves this asymptotics which makes the above partially heuristic reasoning very likely. Moreover, we also tested the two rates of convergence of $\Prob(X\!\in C_n)$ and $\bar d_{n,2}(X)^2$, this time on the joint distribution of the $(W_1, \sup_{t\in [0,1]}W_t)$, $W$ standard Brownian motion which has less symmetries (see appendix \ref{app:num}). They also confirm that the above partially heuristic reasoning is very likely.

\section{Dual quantization: background and basic properties}\label{sec:defs}

Throughout  the paper, except specific mention, $\R^d$ is equipped with a norm $\norm{\cdot}$. 

\subsection{More background}
In the introduction, the definitions related to Voronoi (or regular)and  dual quantizations of a r.v.  $X$ defined on a probability space $(\Omega,{\cal S}, \Prob)$ have been recalled (see~(\ref{def:dnpsharp})-(\ref{def:dbarnp})). 
The aim of this section is to come back briefly to the origin and the
motivations which led us to introduce dual quantization in~\cite{dualStat}. On the way, we will also recall several basic results on dual quantization established in~\cite{dualStat}.
First, we will assume throughout the paper that the r.v.  of interest, $X$, is
truly  $d$-dimensional in the sense that
\[
{\rm aff.dim}({\rm supp}(\Prob_{_X})) =  d.
\] 
 
Let us start by a few practical points.  First note that although all
these definitions are related to a r.v.  $X$, in fact it only depends on the distribution $\Probb = \ProbX$, so we will also often write $\Dqp(\Probb, \Gamma)$ for $\Dqp(X, \Gamma)$ and $\Dqpn(\Probb)$.
%
 Furthermore, to alleviate notations, we will denote from now on $F^p$, $d^p$ and $\bar
 d^p$, \dots instead of  $(F_p)^p$, $(d_p)^p$ and $(\bar d_p)^p$,\dots

\smallskip Let us come back to the terminology {\em dual quantization}: it refers to   a canonical example
of the intrinsic stationary splitting operator: the dual quantization operator.

\smallskip To be more precise,  
let $p\!\in[1,+\infty)$ and let $\Gamma=\{x_1,\ldots,x_n\}\subset \R^d$ be a
grid of size $n\ge d+1$ such that  ${\rm aff.dim}(\Gamma) =  d$ $i.e.$ $\Gamma$ contains at least one $d+1$-tuple of  affinely independent points.

\smallskip The underlying idea is to ``split" $\xi\!\in \conv(\Gamma)$ across at most $d+1$ affinely
independent points in $\Gamma$
proportionally to its barycentric coordinates of $\xi$. There are usually many possible choices of such a $\Gamma$-valued $(d+1)$-tuple of affinely independent points, so  we
introduced a minimal inertia based criterion to select the most appropriate
one  $\xi$, namely the function $F_p(\xi;\Gamma)$ defined for every
$\xi$ as the value of the minimization problem 
\begin{equation}
F_p(\xi;\Gamma)=\inf_{(\lambda_1,\ldots,\lambda_n)}\Big\{\Big(\sum_{i=1}^n
\lambda_i\|\xi-x_i\|^p\Big)^{\frac 1p}, \lambda_i\!\in[0,1] ,
\sum_i\lambda_i\Big[\!\begin{array}{c}x_i\\1\!\end{array}\Big]=\Big[\!\begin{array}{c}\xi\\1\end{array}\!\Big]\Big\}.
\end{equation}
Owing to the compactness of the constraint set ($\lambda_i\ge 0$,
$\sum_i\lambda_i =1$, $\sum_i \lambda_i x_i = \xi$), there exists at least one
solution $\lambda^*(\xi)$  to the above minimization problem. Moreover,  for any such solution,  one shows using convex
extremality arguments, that the set $I^*(\xi):=\big\{i\!\in\{1,\ldots,n\}\mbox{ s.t. } \lambda_i^*(\xi)>0\big\}$ defines an affinely independent subset $\{x_i,\; i\!\in I^*(\xi)\}$.

\smallskip
If, for every $\xi\!\in  conv(\Gamma)$, this solution is  unique, the {\em  dual quantization operator} is simply
defined on $\conv(\Gamma)$ by 
\begin{equation}\label{eq:Jstar}
\forall\, \xi\!\in \conv(\Gamma),\; \forall\,
\omega_0\!\in \Omega_0,\quad \mathcal{J}^*_{\Gamma}(\omega_0, \xi) = \sum_{i\in
I(\xi)^*}x_i \mbox{\bf 1}_{\{\sum_{j=1}^{i-1} \lambda^*_j(\xi)\le U(\omega_0)< 
\sum_{j=1}^{i} \lambda^*_j(\xi)\}},
\end{equation}
where $U$ denotes a  random variable  uniformly distributed over $[0,1]$ on an exogenous probability space  $(\Omega_0,\mathcal{S}_0,
\Prob_0)$. 
This operator ${\cal J}^*_{\Gamma}$ is then measurable
(see~\cite{dualStat}).

The above uniqueness assumption  is not so stringent, especially for applications.
Thus, in a purely Euclidean  quadratic framework:  $\|\,.\,\|= |\,.\,|_{\ell^2}$ (canonical Euclidean norm) and $p=2$  and if $\Gamma$ is said in  ``general position"~(\footnote{no $d+2$ points of $\Gamma$ lie on a sphere in $\R^d$.}), then $\displaystyle\Big\{\{\xi\, \mbox{ s.t. }\, I^*(\xi) =I\}, \, |I|\le d+1\Big\}$ makes up a Borel partition of
$\conv(\Gamma)$ (with possibly empty elements), known in $2$-dimension as the
{\em Delaunay triangulation} of $\Gamma$ (see \cite{rajan} for the connection
with  Delaunay triangulations). 

In a more  general framework, we refer to~\cite{dualStat}
for a construction of dual quantization operators. Such operators are splitting operators since,  by construction, they satisfy the stationarity  property~(\ref{1a}).

\smallskip One must have in mind that the dual quantization operators $\mathcal{J}^*_{\Gamma}(\omega_0, \xi) $   play the  role of the nearest neighbour projections for regular Voronoi quantization. One checks that, by construction, 
\[
\forall\, \xi\!\in \conv(\Gamma),\quad \|\mathcal{J}^*_{\Gamma}(\xi)-\xi\|_{L^p(\Prob_0)}= \|F_p(\xi;\Gamma)\|_{L^p(\Prob_0)}
\]
so that, as soon as  ${\rm supp}(\Prob_{_X})\subset \Gamma$ (or equivalently $\Prob(X\!\in \conv(\Gamma))=1$),
\[
d_{p}(X;\Gamma) = \| \mathcal{J}^*_{\Gamma}(X)-X\|_{L^p(\Prob_0\otimes \Prob)}= \|F_p(X;\Gamma)\|_{L^p(\Prob_0\otimes \Prob)}.
\] 
At this stage, it appears naturally that the the second step of the optimization process is to find (at least) one grid  which optimally
``fits" (the distribution of) $X$ for this criterion $i.e.$ which is the solution to the second
level  optimization problem 
\[ 
d_{n,p}(X)
=\inf\left\{\|\mathcal{J}^*_{\Gamma}(X)-X\|_{L^p(\Prob_0\otimes \Prob)}, \;
\mathcal{J}^*_{\Gamma}: \Omega_0\times \conv(\Gamma)\to \Gamma,
\conv(\Gamma)\supset {\rm supp}(\ProbX),\, |\Gamma|\le n\right\}. 
\] 
Note that
if $X\!\in L_{\R^d}^{\infty}(\Prob)$, $d_{n,p}(X)<+\infty$ if and only if $n\ge d+1$ (whereas it is
identically infinite if $X$ is not essentially bounded). The existence of an
optimal grid (or dual quantizer) has been established in~\cite{dualStat} (see below). 

The error modulus  $d_{n,p}(X)$ can also be characterized as  the {\em lowest $L^p$-mean
approximation error by a r.v.  having at most $n$ values and satisfying the
intrinsic stationarity property} as established in~\cite{dualStat} (Theorem~2, precisely recalled in Theorem~\ref{thm:DQLinkStat} below). 
It should be compared to the well-known property satisfied by the mean (regular) quantization error modulus $e_{n,p}(X)$, namely
\[ 
e_{n,p}(X) = \inf\Big\{\|X-\widehat
X\|_{L^p(\Prob)},\, |\widehat X( \Omega)|\le n\Big\}. 
\]

A stochastic optimization procedure based on a
stochastic gradient approach has been devised in~\cite{dualStat} to  compute optimal dual quantization grids w.r.t. various
distributions  (so far, uniform over $[0,1]^2$, normal, $(W_1,\sup_{t\in[0,1]}  W_t)$, $W$ standard Brownian motion in a purely Euclidean framework). 

Let us conclude by two results established in~\cite{dualStat}. The first one is the   characterization of dual quantization operator in terms  in terms of best
$L^p$-approximation (see~\cite{dualStat}, Theorem~2).

\begin{thm}\label{thm:DQLinkStat}
Let 
$X: \Omega, \mathcal{S},\Prob)\to \R^d$ be a r.v. such that ${\rm aff.dim}({\rm supp}(\Prob_{_X}))= d$ and let  $n\!\in\N$, $n\ge d+1$. Then
\begin{equation*}
\begin{split}
d_{n,p}(X) & = \inf\bigl\{ \E \norm{X -  \JG(X)}_{L^p}: \,
\JG: \Omega_0\times \R^d\to\Gamma, \text{ intrinsic stationary},\\
& \qquad\qquad\qquad\qquad\qquad\qquad\supp(\ProbX) \subset\conv(\Gamma),\,
\abs{\Gamma} \leq n \bigr\}\\
& = \inf\bigl\{ \E \norm{X -  \widehat X}_{L^p}: \widehat X :\PSpace\to \R^d, \\
& \qquad\qquad\qquad\qquad\qquad\qquad \abs{\widehat
X(\Omega_0\times\Omega)} \leq n,\, \E(\widehat X|X) = X \bigr\}\le +\infty.
\end{split}
\end{equation*}
This quantity is finite if and only if $X\in L^\infty(\Omega, \mathcal{S},\Prob)$.
\end{thm}

Finally, the following existence result 
for optimal dual quantizers  {\em at level $n\in\N$}  and the $L^p$-norm
with $p\in(1,\infty)$  is  established in \cite{dualStat}.
 Although we will not
use it in our proofs, this result is recalled for the reader's convenience.

\begin{thm}[Existence of optimal quantizers]\label{thm:existence}
Let $X \in L^p(\Prob)$ for some $p \in (1,\infty)$. 
\begin{enumerate}
  \item[(a)] If $\supp(\ProbX)$ is compact, then  there
  exists for every $n\in \N$ a grid $\Gamman^{\ast} \subset \R^d,\,
  \abs{\Gamman^{\ast}}\leq n$ such that $d_p(X;\Gamman^{\ast}) = d_{n,p}(X)$.
  \item[(b)] If $\ProbX$ is strongly continuous in the sense that it assigns 
  no mass  to  hyperplanes of $\R^d$, then  there exists for every $n\in \N$
  a grid $\Gamman^{\ast} \subset \R^d,\, \abs{\Gamman^{\ast}}\leq n$ such that
$\bar d_p(X;\Gamman^{\ast}) = \bar d_{n,p}(X)$.
\end{enumerate}
If furthermore $\abs{\supp(\ProbX)}\geq n$, then the above statements hold with
$\abs{\Gamman^{\ast}}= n$.
\end{thm}

\subsection{Local properties of the dual quantization functional}\label{sec:localProperties}

We establish or recall  in this paragraph some first general properties of the local $L^p$-dual
quantization functional $F^p$, which will be needed for the final proof of
Theorem \ref{thm:DQRate}.

\begin{prop}\label{prop:PtsInsertion} Let $\Gamma_1$, $\Gamma_2 \subset \R^d$
be finite grids and let $\xi\!\in \R^d$. Then
\[
\Gamma_1 \subset \Gamma_2 \Longrightarrow F_p(\xi; \Gamma_2) \leq F_p(\xi;
\Gamma_1).
\]
\end{prop}
\noindent {\it Proof.} First note that the set $\{\lambda\!\in \R^n\,|\,  \left[  \begin{smallmatrix}
          x_1 & \ldots & x_m\\
          1 & \ldots & 1\\
        \end{smallmatrix}  \right] \lambda =
      \left[\begin{smallmatrix}
         \xi\\ 1\\
        \end{smallmatrix} \right]\}$ 
        is clearly a compact set on which the continuous function 
        $\lambda
        \mapsto \sum_{i=1}^n \lambda_i \|\xi-x_i\|^p$ attains a minimum.
Assume $\Gammaone = \{x_1, \ldots, x_m\}$ and $\Gammatwo = \{x_1, \ldots, x_m,
x_{m+1}, \ldots, x_n\}$. Then
\begin{eqnarray*}
\begin{split}
  F^p(\xi; \Gammatwo)  = \LP{\xi} 
  & \leq
  \underset{ \text{s.t. } \left[  \begin{smallmatrix}
          x_1 & \ldots & x_m\\
          1 & \ldots & 1\\
        \end{smallmatrix}  \right] \lambda =
      \left[\begin{smallmatrix}
         \xi\\ 1\\
        \end{smallmatrix} \right],\,
      \lambda \geq 0 }{\min_{\lambda \in \R^n, \lambda_{m+1}=\cdots= \lambda_n = 0}\sum_{i=1}^n \lambda_i
      \, \norm{\xi-x_i}^p}\\
  & = \underset{ \text{s.t. } \left[  \begin{smallmatrix}
          x_1 & \ldots & x_m\\
          1 & \ldots & 1\\
        \end{smallmatrix}  \right] \lambda =
      \left[\begin{smallmatrix}
         \xi\\ 1\\
        \end{smallmatrix} \right],\,
      \lambda \geq 0 }{\min_{\lambda\in \R^m }\sum_{i=1}^m \lambda_i
      \, \norm{\xi-x_i}^p}
   = F^p(\xi; \Gammaone).\qquad \Box
\end{split}
\end{eqnarray*}

We will also make use of the following three properties established
in~\cite{dualStat} (Propositions~11, 12, 13 respectively). In particular, the third claim yields  a first upper
bound for the asymptotics of the local $L^p$-dual quantization error 
when the size of
the grid goes to infinity.

\begin{prop}\label{prop:rappels}
%
$(a)$ {\em Scalar bound:} Let $\Gamma = \{x_1, \ldots, x_n\}\subset  \R$ with $x_1\leq \ldots\leq x_n$.
Then
\[
 \forall \xi \in [x_1, x_n],\quad \Fp(\xi; \Gamma) \leq \max_{1\leq i \leq n-1}
\Bigl(\frac{x_{i+1}-x_i}{2}\Bigr)^p.
\]
%
$(b)$ {\em Local product Quantization:}
Let $\norm{\cdot} = |\cdot|_{\ell^p}$ 
and let $\Gamma = \Prod_{1\le j\le d} \Gamma_j$ for some $\Gamma_j
\subset \R$. Then 
\[ 
 \forall\, \xi\!\in \R^d, \quad F_{p, |.|_{\ell^p}}(\xi; \Gamma) = \Big(\sum_{j=1}^d F^p(\xi^j;
\Gamma_j)\Big)^{\frac 1p} 
\]
and the same holds true with $\bar F_{p,\ell^p}$ on $\R^d$.
%

$(c)$ {\em Product Quantization:}
Let $C=a+L\,[0,1]^d$, $a=(a_1,\ldots,a_d)\!\in \R^d$, $L>0$, be a hypercube,
with edges parallel to the coordinate axis with common edge-length $L$. Let $\Gamma$ be the product quantizer  of size $(m+1)^d$ defined by
$$
\Gamma=\Prod_{k=1}^d\Big\{a_j+\frac{i L}{m},\, i=0,\ldots,m\Big \}.
$$
%
There exists a positive real  constant  $C_{\norm{.},p} = \sup_{|x|_{\ell^p}=1}\norm{x}^p>0$ such that 
\begin{equation}\label{eq:prodFp}
  \forall \,\xi\!\in C,\quad F^p(\xi; \Gamma)  \leq  d\, C_{\norm{\cdot},p} \cdot
  \Bigl(\frac{L}{2} \Bigr)^p \cdot m^{-p}.
\end{equation}
\end{prop}



\section{Extended Pierce lemma and applications}\label{Pierce}

The aim of this section is to provide a non-asymptotic ``universal" upper-bound for the optimal  (extended) $L^p$-mean 
dual quantization error in the spirit of~\cite{pierce}: it achieves nevertheless the optimal rate of convergence when the size $n$
goes to infinity.
Like for  Voronoi quantization this upper-bound deeply relies on a
random quantization argument and will be a key in the proof of the sharp rate (step~2 of the proof of Theorem~\ref{thm:DQRate}).

For every integer $n\ge 1$, we define the set of ``non-decreasing" $n$-tuples of
$\R^n$ by
\[ {\cal I}_n := \{(x_1,\ldots,x_n)\!\in \R^n,\; -\infty<x_1\le
x_2\le\cdots\le x_n<+\infty\}.
\] 
Let $(x_1,\ldots,x_n)\!\in {\cal I}_n$ (so that $\Gamma=\{x_1,\ldots, x_n\}$ has at most $n$ elements) and  let $\xi\!\in \R$. When $d=1$, it is clear that  the minimization problem~(\ref{eq:Fp}) always
 has a unique solution when $\xi\!\in [x_1,x_n]$ so that, for every $\omega_0\!\in  \Omega_0=[0,1]$, one has 
\begin{eqnarray*}
{\cal \bar J}^*_{(x_1,\ldots,x_n)}(\omega_0,\xi)&=&
\sum_{i=1}^{n-1}\Big(x_i\mbox{\bf 1}_{\{\omega_0 \le \frac{x_{i+1}-\xi}{x_{i+1}-x_i}\}} 
+x_{i+1}\mbox{\bf 1}_{\{\omega_0\ge  \frac{x_{i+1}-\xi}{x_{i+1}-x_i}\}}\Big) \mbox{\bf 1}_{[x_i,x_{i+1})}(\xi)\\
&& +x_1\mbox{\bf 1}_{(-\infty,x_1)}(\xi)+x_n  \mbox{\bf 1}_{[
x_n,+\infty)}(\xi).
\end{eqnarray*}

\smallskip
It follows from~(\ref{def:Fbarp}) that
\begin{eqnarray}
\nonumber\bar F_n^p(\xi, x_1,\ldots,x_n) &=&   \E_{\Prob_0}\big|\xi -{\cal
\bar J}^*_{(x_1,\ldots,x_n)}(\omega_0,\xi)\big|^p\\
\label{eq:Fbard=1}
&=& \sum_{i=1}^{n-1}\left(\frac{(x_{i+1}-\xi)^p(\xi-x_i)}{x_{i+1}-x_i}+\frac{(x_{i+1}-\xi)(\xi-x_i)^p}{x_{i+1}-x_i}\right ) \mbox{\bf 1}_{[x_i,x_{i+1})}(\xi)\\
&& + (x_1-\xi )^p\mbox{\bf 1}_{(-\infty,x_1)}(\xi)+ (\xi-x_n)^p \mbox{\bf 1}_{[
x_n,+\infty)}(\xi)\nonumber
\end{eqnarray}
(the subscript $_n$ is temporarily added to the functional $\bar F^p$, $\bar F_p$, etc,  to emphasize that they are defined on $ {\cal I}_n\times \R$).
The   functionals $\bar F_n^p$ share three important properties   extensively used  in what follows:

\begin{itemize}
\item {\em Additivity}: Let $(x_1, \ldots,x_{i_0},\ldots,x_n)\!\in {\cal I}_n$. Then for every $\xi\!\in \R$
\[
\bar F^p_n (\xi, x_1,\ldots,x_n)= \bar F^p_{i_0} (\xi,
x_1,\ldots,x_{i_0})\mbox{\bf 1}_{(-\infty,x_{i_0})}(\xi)+ \bar F^p_{n- i_0+1}
(\xi, x_{i_0},\ldots,x_{n})\mbox{\bf 1}_{ [x_{i_0}, +\infty)}(\xi).
\]
\item
 {\em  Consistency and monotony}:
 Let $(x_1,\ldots,x_{n})\!\in {\cal I}_{n}$ and $\widetilde x_i\!\in [x_i,x_{i+1}]$ for an $i\!\in \{1,\ldots,n-1\}$. For every  $\xi \!\in \R$, 
\[
\bar F^{p}_{n+1}(\xi, x_1,\ldots,x_{i-1},x_i,\widetilde x_i,x_{i+1},\ldots,x_n)\le \bar F^{p}_{n}(\xi, x_1,\ldots,x_{i-1},x_i,x_{i+1},\ldots,x_n).
\]
When $\xi\!\in [x_1,x_n]$, $\bar F_n^p(\xi;x_1,\dots,x_n)$ coincides with  $
F^p(\xi,\{x_1,\dots,x_n\})$ and this inequality is  a consequence of the
definition of $ F_p$ as the value function of the minimization
problem~(\ref{eq:Fp}). Outside, the above inequality holds as an equality since
it amounts to the nearest distance of $\xi$ to $[x_1,x_n]$. 
%
As a  consequence, 
\begin{equation}\label{monotonie1}
n\longmapsto \bar d_{n,p}(X)=  \inf_{(x_1,\ldots,x_n)\in {\cal I}_n}\|\bar F_{p,n}(X,
x_1,\ldots,x_n)\|_{L^p} \;\mbox{ is non-increasing,}
\end{equation}
More generally, for every fixed $x^0\!\in \R$, both
\begin{equation}\label{monotonie2}
\hskip -0.5cm  n\longmapsto\hskip -0.65cm   \inf_{(x^0,x_2,\ldots,x_n)\in {\cal I}_n}\|\bar F_{p,n}(X,x^0, x_2,\ldots,x_n)\|_{L^p} \,\mbox{ and }\,n\longmapsto \hskip -0.65cm  \inf_{(x_1,x_2,\ldots,x_{n-1},x^0)\in {\cal I}_n}\|\bar F_{p, n}(X, x_1,\ldots,x_{n-1},x^0)\|_{L^p} 
\end{equation}
are non-increasing.

\item {\em Scaling}: $\forall\,\omega\!\in \Omega_0$, $\forall\, (x_1,\ldots,x_n)\!\in {\cal I}_n$, $\forall\, \xi \!\in \R$, $\forall\, \alpha\!\in \R_+$,  $\forall\, \beta\!\in \R$,
\begin{eqnarray*}\label{eq:scaling}
(i)& \bar F^p_n(\alpha\, \xi+\beta, \alpha\, x_1 +\beta ,\ldots,\alpha\, x_n +\beta)&= \alpha \,\bar F_n^p(\xi,x_1,\ldots,x_n),\\
(ii) &\bar F^p_n(\xi, x_1   ,\ldots,x_n)&=\bar
F^p_n(-\xi,-x_n,\ldots,-x_1).
\end{eqnarray*}
\end{itemize}


\begin{thm}\label{Pierce0}Let $p,\,\eta>0$. There exists a real
constant $C_{p,\eta}>0$
such that for every random variable  $X:(\Omega,{\cal A},\Prob)\to \R$,
\[
\forall\, n\ge 1,\;\quad
\inf_{(x_1,\ldots,x_n)\in {\cal I}_n}\|\bar
F_{p,n}(X, x_1,\ldots,x_n)\|_{L^p} \le C_{p,\eta}\|X\|_{L^{p+\eta}}n^{-1}.
\] 
\end{thm}

The proof below relies  on a random quantization argument involving an $n$-sample of the Pareto$(\delta)$-distribution on $[1,+\infty)$. Though significantly more demanding, it plays the same  crucial role in establishing the sharp rate result as the so-called Pierce Lemma established in~\cite{meanRegular} (see also~\cite{Foundations}) for Voronoi quantization    to prove the original Zador Theorem.  

In the proof, we will make use of the $\Gamma$ and $B$ functions defined by
$\Gamma(a)=\int_0^{+\infty}u^{a-1}e^{-u}du$, $a>0$, and
$B(a,b) =\int_0^1u^{a-1}(1-u)^{b-1}du$, $a,b>0$,  respectively, and satisfying $B(a,b)= \frac{\Gamma(a)\Gamma(b)}{\Gamma(a+b)}$.

\begin{proof} 
%
{\sc Step~1.} 
%
We first assume that $X$ is $[1,+\infty)$-valued and $n\ge 2$. Let $(Y_n)_{n\ge1}$ be a sequence of i.i.d. Pareto$(\delta)$-distributed random variables (with probability density $f_{_Y}(y) =\delta y^{-\delta-1}\mbox{\bf 1}_{\{y\ge 1\}}$) defined on a probability space $(\Omega',{\cal A}', \Prob')$. 
%
%

\smallskip Let $\delta=\delta(\eta,p)\!\in (0,\frac{\eta}{\lceil p \rceil})$ be chosen so that $\ell=\ell(p,\eta)=\frac{p}{\delta}$  is an integer and $\ell\ge 2$. 
 For every $n\ge \ell(p,\eta)$, set $\widetilde n= n-\ell+2\!\in \N$, $\widetilde n\le n$. It follows from the monotony property~(\ref{monotonie2}) that
\begin{eqnarray*}
\inf _{(1, x_2,\ldots,x_n)\in {\cal I}_n}\| \bar F_{p,n} (X, 1,x_2,\ldots,x_n) \|_{L^p}& \le & \hskip -0.25 cm \inf_{(1,x_2,\ldots,x_{\widetilde n})\in {\cal I}_{\widetilde n}} \hskip -0.25 cm\|\bar F_{p,\widetilde n }(X,1,x_2,\ldots,x_{\widetilde n})\|_{L^p}\\
&\le& \hskip -0.25 cm\|\bar F_{p,\widetilde n}(X,Y^{(n)}_0,Y^{(n)}_1,\ldots,Y^{(n)}_{\widetilde n-1})\|_{L^p(\Omega\times \Omega', \Prob\otimes \Prob')}
\end{eqnarray*}

\noindent where, for every $n \ge 1$,  $Y^{(n)}=(Y^{(n)}_1,\ldots,Y^{(n)}_n)$ denotes the standard order statistics of the first $n$  terms of the sequence $(Y_k)_{k\ge1}$ and $Y^{(n)}_0=1$.  On the other hand, we recall (see $e.g.$~\cite{CAVC}) that  the joint distribution of $(Y^{(n)}_i,Y^{(n)}_{i+1})$, $1\le i\le n-1$,  is given by 
\[
\Prob'_{(Y^{(n)}_i,Y^{(n)}_{i+1})}(du,dv)=\delta^2  \frac{n!}{(i-1)!(n-i-1)!}(1-u^{-\delta})^{i-1}v^{-\delta(n-i-1)}(uv)^{-\delta-1}\,du\,dv.
\]


\medskip
\noindent {\sc Step~2.} Assume that $n\ge 3$.  Since $X$ and $(Y_1,\ldots,Y_0)$ are  independent and $X\ge 1$
\[
\|\bar F_{p,\widetilde n}(X,Y^{(n)}_0,Y^{(n)}_1,\ldots,Y^{(n)}_{\widetilde n-1})\|^p_{L^p(\Omega\times \Omega', \Prob\otimes \Prob')}= \int_{[1,+\infty)}\|\bar F_{p,\widetilde n}(\xi, Y^{(n)}_0,Y^{(n)}_1,\ldots,Y^{(n)}_{\widetilde n-1})\|^p_{L^p( \Omega', \Prob')}\Prob_{_X}(d\xi).
\]
Relying on the expression~(\ref{eq:Fbard=1}) of the functional $\bar F^p_n$, we set for every $i=0,\ldots,n-\ell$ and $\xi\ge 1$
 \[
 (a)_i := \E \left(\frac{(Y^{(n)}_{i+1}-\xi)^p(\xi-Y^{(n)}_i)}{Y^{(n)}_{i+1}-Y^{(n)}_i}\mbox{\bf 1}_{\{Y^{(n)}_i< \xi\le Y^{(n)}_{i+1}\}}\right),\,
 (b)_i := \E \left(\frac{(Y^{(n)}_{i+1}-\xi)(\xi-Y^{(n)}_i)^p}{Y^{(n)}_{i+1}-Y^{(n)}_i}\mbox{\bf 1}_{\{Y^{(n)}_i< \xi\le Y^{(n)}_{i+1}\}}\right)
 \]
 and $\displaystyle (c)_{\widetilde n-1}:= \E\Big(\big ( \xi-Y^{(n)}_{n-\ell+1}\big)^p\mbox{\bf 1}_{\{\xi\ge Y^{(n)}_{n-\ell+1}\}} \Big)$.
 
 We will first inspect the sum $\sum_{i=0}^{n-\ell} (\Box)_i$, $\Box=a,b$   successively.
 
 Let $i\!\in \{1, \ldots, \tilde n-1\}$. It follows from the above expression of the distribution of $(Y^{(n)}_i,Y^{(n)}_{i+1})$ that
\[
 (a)_i = \delta^2 \int\!\!\int_{1\le u\le \xi\le v} \frac{(v-\xi)^p(\xi-u)}{v-u}(1-u^{-\delta})^{i-1}v^{-\delta(n-i-1)}(uv)^{-\delta-1} \,du\,dv \frac{n!}{(i-1)!(n-i-1)!}.
\]
The change of variable $v= \xi(w+1)$   yields
\[
(a)_i =  n(n-1)\left(\begin{smallmatrix} n-2\\i-1\end{smallmatrix}\right)\delta^2 \int_1^{\xi} \hskip -0.15cm du\, (\xi-u)(1-u^{-\delta})^{i-1} u^{-\delta-1} \xi^{p-\delta(n-i)} \hskip -0.15cm  \int_{0}^{+\infty} \hskip -0.15cm  dw\,  \frac{w^p}{\xi(w+1)-u}(w+1)^{-\delta(n-i)-1}.
\]
Noting that $\frac{\xi-u}{\xi(w+1)-u}\le \frac{1}{w+1}$ then leads to
\[
(a)_i \le   n(n-1)\left(\begin{smallmatrix} n-2\\i-1\end{smallmatrix}\right)\delta^2n(n-1)\xi^{p-\delta (n-i)}\int_1^{\xi} \!(1-u^{-\delta})^{i-1} u^{-\delta-1}du \,\times \int _0^{+\infty} \! w^p(1+w)^{-\delta(n-i)-2}dw .
\]
The change of variable $w=\frac{1}{y}-1$ shows that $\displaystyle  \int _0^{+\infty} w^p(1+w)^{-\delta(n-i)-2}dw  = B(\delta(n-i)-p+1, p+1)$ whereas $\displaystyle \int_1^{\xi}  (1-u^{-\delta})^{i-1} u^{-\delta-1} du = \frac{(1-\xi^{-\delta})^i}{\delta i}$ so that
\[
(a)_i \le  \delta n \left(\begin{smallmatrix} n-1\\i\end{smallmatrix}\right) (1-\xi^{-\delta})^i \xi^{p-\delta (n-i)} \frac{\Gamma(p+1)\Gamma(\delta(n-i)-p+1)}{\Gamma(\delta (n-i)+2)}
\]
where we used the standard  identity 
$ \left(\begin{smallmatrix} n-1\\i\end{smallmatrix}\right)= \frac{n-1}{i}\left(\begin{smallmatrix} n-2\\i-1\end{smallmatrix}\right)$. 

When $i=0$, noting that the density of $Y^{(n)}_1= \min_{1\le i\le n}Y_i$ is
$\delta n y^{-\delta n-1}\mbox{\bf 1}_{\{y\ge 1\}}$, we get
 \begin{eqnarray*}
(a)_0&= &\E \left(\frac{(Y^{(n)}_1-\xi)^p(\xi-1)}{Y^{(n)}_1-1}\mbox{\bf 1}_{\{1\le \xi\le Y^{(n)}_1\}}\right) \\
&=& \delta n \int_{\xi }^{+\infty} (\xi-1)\frac{(v-\xi )^p}{v-1} v^{-\delta n-1} dv\\
&=& \delta n \xi^{p-\delta n}\int_0^{+\infty}\frac{(\xi-1)}{\xi(w+1)-1}  w^p (w+1)^{-\delta n-1} dw \quad \mbox{where we set  $v=\xi (w+1)$}  \\
&\le &   \delta n \xi^{p-\delta n} B(\delta n -p+1,p+1)
  \end{eqnarray*}
where we used in the last line that $\frac{\xi-1}{\xi(w+1)-1} \le \frac{1}{w+1}$.   As a consequence 
 \begin{eqnarray*}
\sum_{i=0}^{n-\ell} (a)_i &\le &  \delta \,n\, \Gamma(p+1) \sum_{i=0}^{n-\ell} \left(\begin{smallmatrix} n-1\\ i \end{smallmatrix}\right)\xi^{p-\delta(n-i)} (1-\xi^{-\delta})^i\frac{\Gamma(\delta(n-i)-p+1}{\Gamma(\delta(n-i)+2}\\
&\le  &\delta\, n\, \Gamma(p+1)\xi^p (1-\xi^{-\delta})^n\sum_{j=\ell}^n \left(\begin{smallmatrix} n-1\\ j -1\end{smallmatrix}\right)(\xi^{\delta}-1)^{-j} \frac{\Gamma(\delta j-p+1)}{\Gamma(\delta j +2)}.
\end{eqnarray*}
Now using that for every $a>0$, $\frac{\Gamma(x+a)}{\Gamma(x)}\sim x^a$ as $x\to \infty$, we derive the existence of a real constants $\tilde \kappa^{(0)}_{p,\delta}, \,\kappa^{(0)}_{p,\delta}>0$ such that
\[
\forall\, j\ge 0,\quad  \frac{\Gamma(\delta j-p+1)}{\Gamma(\delta j +2)} \le \tilde \kappa^{(0)}_{p,\delta}\, j^{-(p+1)}\le  \kappa^{(0)}_{p,\delta} \, \frac{j^{\lceil p\rceil -p}}{j(j+1)\cdots(j+\lceil p\rceil )}.
\]
In turn, using that
\[
\left(\begin{smallmatrix} n+\lceil p\rceil \\ j+\lceil p\rceil  \end{smallmatrix}\right) = \frac{(n+\lceil p\rceil )\cdots n}{(j+\lceil p\rceil ) \cdots j} \left(\begin{smallmatrix} n-1\\ j-1 \end{smallmatrix}\right),
\]
we finally obtain
 \begin{eqnarray*}
\sum_{i=0}^{n-\ell} (a)_i& \le& \kappa^{(0)}_{p,\delta}\,  n \Gamma(p+1)\xi^p
\delta  (1-\xi^{-\delta})^n \frac{1}{(n+\lceil p\rceil )\cdots
(n+1)n}\sum_{j=\ell}^{n} \left(\begin{smallmatrix} n+\lceil p\rceil \\ j+\lceil p\rceil  \end{smallmatrix}\right)(\xi^{\delta}-1)^{-j}j^{\lceil p\rceil -p}\\
&\le&  \kappa^{(0)}_{p,\delta} \, \Gamma(p+1)\xi^p \delta  (1-\xi^{-\delta})^n \frac{n^{\lceil p\rceil -p}}{(n+\lceil p\rceil )\cdots (n+1)}(\xi^{\delta}-1)^{\lceil p\rceil }\Big(1+(\xi^{\delta}-1)^{-1}\Big)^{n+\lceil p\rceil }.  
\end{eqnarray*}
Now 
\[
 (1-\xi^{-\delta})^n \xi^p(\xi^{\delta}-1)^{\lceil p\rceil} \Big(1+(\xi^{\delta}-1)^{-1}\Big)^{n+\lceil p\rceil }= \xi^{p + \delta\lceil p\rceil}
\]
%
so that, using that $\xi\ge 1$ and $ \delta <\frac{\eta}{\lceil p\rceil}$,  we get $\xi^{p + \delta\lceil p\rceil}
\le  \xi^{p+\eta}$ which in turn implies   
\[
\sum_{i=0}^{n-\ell} (a)_i  \le  \kappa^{(0)}_{p,\delta}  \, \delta\,   \Gamma(p+1) \xi^{p+\eta} \frac{1}{n^p}.
\]

%
Let us pass now to the second sum involving $(b)_i$.  First note that, on the event $\displaystyle\Big \{Y^{(n)}_i \le \xi\le \frac{Y^{(n)}_i+Y^{(n)}_{i+1}}{2}\Big\}$ (which is clearly included in $\big\{Y^{(n)}_i \le \xi\le Y^{(n)}_{i+1}\big\}$), one has $(\xi-Y^{(n)}_i)^p(Y^{(n)}_{i+1}-\xi)\le (\xi-Y^{(n)}_i)(Y^{(n)}_{i+1}-\xi)^p$
 so that, owing to what precedes, we can focus on $\displaystyle \sum_{i=0}^{n-\ell} (\widetilde b)_i$ where 
 \[
  (\widetilde b)_i:=\E \left((\xi-Y^{(n)}_i)^p\mbox{\bf 1}_{\big\{  \frac{Y^{(n)}_i+Y^{(n)}_{i+1}}{2}\le \xi\le Y^{(n)}_{i+1}\big\}}\right)\ge \E \left(\frac{(Y^{(n)}_{i+1}-\xi)(\xi-Y^{(n)}_i)^p}{Y^{(n)}_{i+1}-Y^{(n)}_i}\mbox{\bf 1}_{\big\{  \frac{Y^{(n)}_i+Y^{(n)}_{i+1}}{2}\le \xi\le Y^{(n)}_{i+1}\big\}}\right).
 \]
 
This time we will analyze  successively the sum over $i=1,\ldots,n-\ell$ and the case $i=0$.
\begin{eqnarray*}
\sum_{i=1}^{n-\ell} (\widetilde b)_i&=& \delta^2 n(n-1)
\int\!\!\int_{\{1\le u\le \xi\le v\le 2\xi-u\}}\hskip -0.5 cm du\,dv\,
(uv)^{-\delta-1}(\xi-u)^p \sum_{i=1}^{n-\ell} \left(\begin{smallmatrix} n-2\\ i-1
\end{smallmatrix}\right)v^{-\delta(n-2-(i-1))} (1-u^{-\delta})^{i-1}
\\
&\le& \delta^2 n(n-1) \int\!\!\int_{\{1\le u\le \xi\le v\le
2\xi-u\}}du\,dv(uv)^{-\delta-1}(\xi-u)^p(1-u^{-\delta} +v^{-\delta})^{n-2}\\
&\le &  \delta^2 n(n-1) \int_1^{\xi} du u^{-\delta-1}(\xi-u) ^p  \int_{ 
\xi}^{2\xi-u} dv\, v^{-\delta-1}e^{-(n-2)(u^{-\delta}-v^{-\delta})}\\
&=&  \delta^2 n(n-1) \int_1^{\xi} du u^{-\delta-1}(\xi-u) ^p
e^{-(n-2)u^{-\delta}} \int_{  \xi}^{2\xi-u} dv\, v^{-\delta-1}e^{(n-2)v^{-\delta}}
\end{eqnarray*}
where we used in the  in the second line that $n-\ell-1\le n-2$ since $\ell\ge 1$. Setting $v= y^{-\frac{1}{\delta}}$ yields
\begin{eqnarray*}
 \int_{  \xi}^{2\xi-u}  v^{-\delta-1}e^{(n-2)v^{-\delta}}\, dv&=& \frac{1}{\delta} \int_{(2\xi-u)^{-\delta}}^{\xi^{-\delta}}e^{(n-2)y}dy\\
 &\le& \frac{1}{\delta} \big(\xi^{-\delta} -(2\xi-u)^{-\delta}\big)e^{(n-2)\xi^{-\delta}}\\
 &\le& (\xi-u) \xi^{-\delta-1}e^{(n-2)\xi^{-\delta}}
\end{eqnarray*}
where we used in the last line  the fundamental formula of Calculus. Consequently, 
\begin{eqnarray*}
\sum_{i=1}^{n-\ell} (\widetilde b)_i&\le&  n(n-1) \delta^2 \xi^{-\delta-1}\int_1^{\xi} u^{-\delta-1} (\xi-u)^{p+1} e^{-(n-2)(u^{-\delta}-\xi^{-\delta})}du\\
&=& n(n-1) \xi^{-\delta-1} \delta \int_0^{(n-2)(1-\xi^{-\delta})}  \Big(\xi-\big(\frac{x}{n-2}+\xi^{-\delta}\big)^{-\frac {1}{\delta}}\Big)^{p+1} e^{-x}\frac{dx}{n-2}
\end{eqnarray*}
where we   put $u=\big(\frac{x}{n-2}+\xi^{-\delta}\big)^{-\frac{1}{\delta}}$.
Now, applying again fundamental formula of Calculus to the function
$z^{-\frac{1}{\delta}}$ yields,
\[
\xi-\Big(\frac{x}{n-2}+\xi^{-\delta}\Big)^{-\frac {1}{\delta}} = (\xi^{-\delta})^{-\frac{1}{\delta}}-\Big(\frac{x}{n-2}+\xi^{-\delta}\Big)^{-\frac {1}{\delta}}\le \frac{x}{\delta(n-2)}\xi^{\delta+1}
\]
\begin{eqnarray*}
\mbox{so that  }\quad \sum_{i=1}^{n-\ell} (\widetilde b)_i & \le&   \frac{n(n-1)}{(n-2)^{p+2}} \delta^{-p} \xi^{(p+1)(\delta+1)-(\delta+1)} \int_0^{(n-2)(1-\xi^{-\delta})} x^{p+1} e^{-x}dx \\
&\le & \kappa^{(1)}_{p,\delta} \Gamma(p+2) n^{-p} \xi^{p(\delta+1)}
\end{eqnarray*}
for some constant $\kappa^{(1)}_{p,\delta} > 0$.

When $i=0$, keeping in mind that $Y^{(n)}_1 = \min_{1\le i\le n}Y_i$, 
\begin{eqnarray*}
(\widetilde b)_0& \le& (\xi-1)^p \Prob(\xi \le Y^{(n)}_1\le 2\xi-1)= (\xi-1)^p \big(\xi^{-n\delta}-(2\xi-1)^{-n\delta}\big) \\
&\le& n\delta (\xi-1)^{p+1} \xi^{-n\delta-1}=  n\delta \xi^{p(1+\delta)} g(1/\xi)
\end{eqnarray*}
where $g(u) = (1-u)^{p+1}u^{(n+p)\delta}$, $u\!\in (0,1)$. One checks that $g$ attains its maximum over $(0,1]$ at $u^*= \frac{(n+p)\delta}{(n+p)\delta +p+1}$
so that
\[
\sup_{u\in (0,1]}g(u) = g(u^*)= \left(\frac{p+1}{(n+p)\delta +p+1}\right)^{p+1}(u^*)^{(n+p)\delta} \le  \left(\frac{1}{1+\frac{n+p}{p+1}\delta }\right)^{p+1}.
\]
Finally, there exists a real constant $\kappa^{(2)}_{p,\delta}>0$ such that 
\[
(\widetilde b)_0 \le \xi^{p(\delta+1)} \frac{\delta n}{(1+\frac{n+p}{p+1}\delta)^{p+1}}\le \kappa^{(2)}_{p,\delta}  \xi^{p(\delta+1)}  n^{-p}.
\]

As concerns the $(c)_{n-\ell+1}$ term, we proceed as follows. 
\begin{eqnarray*}
\E\Big(\big ( \xi-Y^{(n)}_{n-\ell+1}\big)^p\mbox{\bf 1}_{\{\xi \ge Y^{(n)}_{n-\ell+1}\}} \Big)&\le&  \xi ^p \Prob(\xi \ge Y^{(n)}_{n-\ell+1})\\
&\le& \xi^{p(1+\delta)}  \E (Y^{(n)}_{n-\ell+1})^{-p\delta}.
\end{eqnarray*}
Note that 
\begin{eqnarray*}
\E (Y^{(n)}_{n-\ell+1})^{-p\delta} &=&\frac{\Gamma(n+1)}{\Gamma(n-\ell+1)\Gamma(\ell)}\int_0^1 (1-v)^{n-\ell}v^{\ell+p-1}dv =\frac{\Gamma(n+1)}{\Gamma(\ell)} \frac{\Gamma(\ell+p)  }{\Gamma(n+p+1)}\\
&\sim& \frac{\Gamma(\ell+p)}{\Gamma(\ell)}n^{-p} = O(n^{-p}).
\end{eqnarray*}
Finally, for every $ \xi \ge 1$, 
\[
(c)_{n-\ell+1}\le \kappa_{p,\delta}^{(3)}\xi^{p(1+\delta)}n^{-p}.
\]
Consequently, there exists a real constant  $\kappa_{p,\eta} =
\max_{j=0,\ldots,3} \kappa^{(j)}_{p, \delta}>0$ such that for
every $n\ge n_{p,\eta}= \ell(\eta,p)\vee 3$,
\[
 \forall\, \xi \ge 1,\qquad n^p \,\E\,
 \bar F^p_{n}(\xi, Y^{(n)}_0,\ldots,Y^{(n)}_{\widetilde n+1})\le
 \kappa_{p,\eta} \,\xi^{p+\eta}
\]
since $p\,\delta \le \eta$. Hence for every r.v.  $X$, we derive by integrating in $\xi\!\in [1,+\infty)$ with respect to $\Prob_{_X}(d\xi)$:
\[
n^p \inf_{(1,x_2,\ldots,x_n)\in {\cal I}_n}\E\, \bar F^p_{n}(X,1,x_2,\ldots,x_n)\le  n^p \E\, \bar F^p_{n}(X,Y^{(n)}_0,\ldots,Y^{(n)}_{\widetilde n+1})\le \kappa_{p,\eta}\,\E \,X^{p+\eta}.
\]

 \noindent {\sc Step~3.} If $X$ is a non-negative random variable, applying the second step to $X+1$ and using the scaling property $(i)$ satisfied by $F_{p,n}$ yields for $n\ge  n_{p,\eta}$ (as defined in Step~2), 
\begin{eqnarray*} 
 \inf_{ (0,x_2,\ldots,x_n)\in {\cal I}_n}\|\bar F_{p,n}(X,0,x_2,\ldots,x_n)\|_{L^p} &=& \inf_{ (1,x_2,\ldots,x_n)\in {\cal I}_n}\|\bar F_{p,n}(X+1,1,\ldots,x_n)\|_{L^p} \\
 &\le&  \kappa^{1/p}_{p,\eta}\frac{\|1+X\|_{L^{p+\eta}}^{1+\frac{\eta}{p}}}{n}\\
 &\le&  C^{(0)}_{p,\eta}\frac{(1+ \|X\|_{L^{p+\eta}}^{1+\frac{\eta}{p}})}{n}\; \mbox{ with } C^{(0)}_{p,\eta}= (2^{1+\eta}\kappa_{p,\eta})^{\frac 1p}.
 \end{eqnarray*}
 We may assume that $\|X\|_{L^{p+\eta}}\!\in(0,\infty)$. Then, applying the above bound to the non-negative random variable $\widetilde X= \frac{X}{\|X\|_{L^{p+\eta}}}$ taking again advantage of the scaling property $(i)$, we obtain
 \begin{eqnarray*}
  \inf_{ (0,x_2,\ldots,x_n)\in {\cal I}_n}\|\bar
  F_{p,n}(X,0,x_2,\ldots,x_n)\|_{L^p}&=&  \|X\|_{L^{p+\eta}} \inf_{
  (0,x_2,\ldots,x_n)\in {\cal I}_n}\|\bar F_{p,n}(\widetilde
  X,0,x_2,\ldots,x_n)\|_{L^p}\\  &\le & \|X\|_{L^{p+\eta}}
  C^{(0)}_{p,\eta}\frac{1+1}{n} = 2C^{(0)}_{p,\eta}
  \,\|X\|_{L^{p+\eta}} \frac{1}{n}.
 \end{eqnarray*}
%

 \noindent {\sc Step~4.} Let $X$ be a real-valued random variable  and let for every integer $n\ge 1$, $x_1,\ldots,x_n\!\in (-\infty,0)$,  $x_{n+1}=0$  and $x_{n+2},\ldots,x_{2n+1}\!\in (0,+\infty)$. It follows from the additivity property that  that 
\begin{eqnarray*}  
 \nonumber \bar F^p_{2n+1}(X, x_1,\ldots,x_{2n+1}) &=&\bar F^{p}_{n+1}(X_+, x_{n+1},\ldots,x_{2n+1})\mbox{\bf 1}_{\{X\ge 0\}}  \\
 &&+  \bar F^{p}_{n+1}(-X_-, x_{1},\ldots,x_{n+1})\mbox{\bf 1}_{\{X<0\}} \\
\label{IneqA} &=& \bar F^{p}_{n+1}(X_+, x_1,\ldots,x_{n+1})\mbox{\bf 1}_{\{X
\ge 0\}}   +  \bar F^{p}_{n+1}(X_-, -x_{n+1},\ldots,-x_{1})\mbox{\bf 1}_{\{X<0\}}.
 \end{eqnarray*} 
Consequently, using that $X_+\times X_- \equiv 0$ and that $x_{n+1}=0$, we get 
\begin{eqnarray*}  
  \!\!\!\!\inf_{\substack{(x_1,\ldots,x_{2n+1})\in {\cal
  I}_{2n+1}\\ x _{n+1}=0}}\|\bar F_{p,2n+1}(X,
  x_1,\ldots,x_{2n+1})\|^p_{L^p} \!\!&\!\!\le\!\!&\!\!  \inf_{ (0,y_2,\ldots,y_{n+1})\in {\cal I}_{n+1}} \|\bar F_{p,n+1}(X_+,0,y_2,\ldots,y_{n+1}) \|^p_{L^p} \\
  &&+ \inf_{ (0,y_2,\ldots,y_{n+1})\in {\cal I}_{n+1}}\|\bar F_{p,n}(X_-,0,y_2,\ldots,y_{n+1}) \|^p_{L^p}  .
  \end{eqnarray*} 
   Hence,  it follows from Step~2 that, for every  $n\ge n_{p,\eta}-1$,
\begin{eqnarray*}  
 \inf_{ (x_1,\ldots,x_{2n+1})\in {\cal I}_{2n+1}}\|\bar
 F_{p,2n+1}(X,x_1,\ldots,x_{2n+1})\|^p_{L^p}   &\le & \Big(
 \|X_-\|^p_{L^{p+\eta}} +\|X_+\|^p_{L^{p+\eta}}
 \Big)\Big(\frac{2C^{(0)}_{p,\eta} }{n+1}\Big)^p.
 \end{eqnarray*} 
 Now using that $(a+b)\le2^{1-\frac 1q} (a^q+b^q)^{\frac 1q}$, $a,b\ge 0$, with $q=1+\frac{\eta}{p}\ge 1$, we derive that
 \[
  \|X_-\|^p_{L^{p+\eta}} +\|X_+\|^p_{L^{p+\eta}} \le 2^{\frac{\eta}{p+\eta}} \Big(\|X_-\|^{p+\eta}_{L^{p+\eta}} +\|X_+\|^{p+\eta}_{L^{p+\eta}} \Big)^{\frac{p}{p+\eta}}= 2^{\frac{\eta}{p+\eta}}  \|X\|_{L^{p+\eta}}^p
 \]
 since $X_-\times X_+\equiv 0$. 
Now, the monotonicity property~(\ref{monotonie1})  implies that, for every $n\ge 2\,n_{p,\eta}$,
 \[
 \bar d_{n,p}(X)=  \inf_{ (x_1,\ldots,x_{n})\in {\cal I}_{n}}\|\bar F_{p,n}(X,x_1,\ldots,x_{n})\|_{L^p}\le 2^{\frac{\eta}{p(p+\eta)}}2C^{(0)}_{p,\eta}\frac{\|X\|_{L^{p+\eta}} }{n}.
 \]
Still calling upon~(\ref{monotonie1}), we note that, for every $n\!\in\{1,\ldots,2n_{p,\eta}\}$, $\bar d_{n,p}(X)\le \bar d_{1,p}(X) = \inf_{x \in \R}\|X-x_1\|_{L^p}\le \|X\|_{L^p}$ so that
\[
\bar d_{n,p}(X)\le 2n_{p,\eta} \frac{ \|X\|_{L^{p+\eta}}}{n}
\]
which completes the proof by setting $C_{p,\eta} = \max\big(2n_{p,\eta},2^{1+\frac{\eta}{p(p+\eta)}}C^{(0)}_{p,\eta}\big)$.
%
 \end{proof}

 \subsection{A $d$-dimensional non-asymptotic upper-bound for the dual
 quantization error}\label{sec:ddimBound}

Now, combining Theorem~\ref{Pierce0} and Proposition~\ref{prop:rappels}$(b)$, we are in position to show Proposition~\ref{PdtQErrop} (the
$d$-dimensional version of the extended Pierce Lemma)  which provides a non-asymptotic upper-bound at the exact rate for dual quantization error moduli.

%
 
\medskip
\noindent {\it Proof of Proposition~\ref{PdtQErrop}.} $(a)$ First note that $ \bar d_{n,p}(X)= \bar d_{n,p}(X-a)$, $a\!\in \R^d$ (invariance by translation)    so we may assume that $X$  is $L^{p+\eta}$-centered $i.e.$ $\sigma_{p+\eta,\|.\|}(X)= \|X\|_{L^{p+\eta}}$. When $d=1$, Theorem~\ref{Pierce0} solves the problem.

%
%
Let $d\ge 2$. Let  $X=(X^1,\ldots,X^d)$ ($X^i$ components of $X$).   It follows form Proposition~\ref{prop:rappels}
 that, if $\Gamma=\prod_{1\le i \le d} \Gamma_i$, with $\Gamma_i\subset \R$, $|\Gamma_i|=n_i$ with $n_1\cdots n_d\le n$. Then for every $\xi=(\xi^1,\ldots,\xi^d)\!\in \R^d$
 \[
  \bar F^p_{\norm{.}}(\xi;\Gamma) \le C_{p, \norm{.}}\bar F^{p}_{\ell^p}(\xi;\Gamma) = \sum_{j=1}^d\bar F^p(\xi^j, \Gamma_j) 
 \]
where $C_{p,\norm{.}}\!=\! \sup_{|\xi|_{\ell^p}=1}\|\xi\|^p$. Integrating with respect to the distribution of $X$ yields  $\displaystyle \bar d^p (X,\Gamma)\!\le C_{p, \norm{.}} \sum_{j=1}^d \bar d^p(X^j, \Gamma_j)$ 
which in turn easily implies  
\[
\bar d_n^p(X)\le C_{p, \norm{.}} \sum_{j=1}^d \bar d_{n_j}^p(X^j). 
\]
Now set $n_j = \lfloor n^{\frac 1d}\rfloor$, $j=1,\ldots,d$. It follows from Theorem~\ref{Pierce0} that
%
 \begin{eqnarray*}
 \bar  d_{n}^p(X) & \le &  C^p_{p, \norm{.}}C_{p,\eta}  \sum_{j=1}^d \|X^{j}\|^p_{L^{p+\eta}}\lfloor n^{\frac 1d}\rfloor^{-p} \\
 & \le &  C_{p, \norm{.}} C_{p,\eta} \sup_{k\ge 2}\Big(\frac{k^{\frac 1d}}{k^{\frac 1d}-1}\Big)^p n^{-\frac pd} \sum_{j=1}^d \|X^{j}\|^p_{L^{p+\eta}}\\
& \le &  C_{p, \norm{.}} C_{p,\eta} 2^p n^{-\frac pd}d^{\frac{\eta}{p+\eta}}\E |X|^{p+\eta}_{\ell^{p+\eta}}\\
&\le&  d^{\frac{\eta}{p+\eta}} C_{p, \norm{.}} C_{p,\eta} \,2^p \widetilde C_{\norm{.}, p+\eta}  \|X\|_{L^{p+\eta}}^{p+\eta}n^{-\frac pd}
 \end{eqnarray*}
 where $\widetilde C_{\norm{.}, r}=\sup_{\|x\|=1}|x |^r_{\ell^r}$, $r>0$. 
 
\smallskip
\noindent $(b)$ Let $C$ be  the smallest hypercube withe edges
parallel to the coordinate axis containing $\conv({\rm Supp}(\Prob_{_X}))$. Up
to a translatation, which  leaves $d_{n,p}(X)$ invariant, we may assume that $C=[0,L]^d$ where $0\le L\le {\rm diam}_{\norm{.}}({\rm Supp}(\Prob_{_X}))$. The conclusion follows by integrating Inequality~(\ref{eq:prodFp}) with respect to $\Prob_{_X}(d\xi)$ with $m=\lfloor n^{\frac 1d} \rfloor$ and following the lines of the proof of claim~$(a)$. $\qquad \Box$ 

\section{Proof of the sharp rate theorem}\label{sec:rate}

On the way to proving the sharp rate theorem, we have to establish few additional 
propositions.

\begin{prop}[Sub-linearity]\label{prop:subLin}
Let $\Probb = \sum_{i=1}^m s_i \Probb_i$ where $s_1,\ldots,s_m\!\in [0,1]$,  $\sum_{i=1}^m s_i = 1$ and let $n_1,\ldots,n_m\!\in \N$ such that $\sum_{i=1}^m n_i \leq n$.
Then
\[
d_{n}^p(\Probb) \leq \sum_{i=1}^m s_i\, d_{n_i}^p(\Probb_i).
\]
\end{prop}
\begin{proof}
For $\varepsilon > 0$ and every $i = 1, \ldots, m$, choose $\Gammai \subset
\R^d,\,\abs{\Gammai} \leq n_i$ such that 
\[
d^p(\Probb_i; \Gammai) \leq (1+\varepsilon)\, d_{n_i}^p(\Probb_i).
\]
%
Set $\Gamma =
\bigcup_{i=1}^m \Gammai$~; from Proposition~\ref{prop:PtsInsertion} we get
\begin{eqnarray*}
\begin{split}
  d_{n}^p(\Probb) & \leq d_{n}^p(\Probb; \Gamma) = \sum_{i=1}^m s_i \int F^p(\xi; \Gamma)\,\Probb_i(d\xi)\\
  & \leq \sum_{i=1}^m s_i \int F^p(\xi; \Gammai)\,\Probb_i(d\xi)  \leq (1+\varepsilon) \sum_{i=1}^m s_i\, d_{n_i}^p(\Probb_i).
\end{split}
\end{eqnarray*}
Letting $\varepsilon\to 0$ completes the proof.
\end{proof}

\begin{remark}
Proposition~\ref{prop:subLin} does not hold for $\bar d_n^p$ since $\bar F^p$ is not decreasing for the inclusion order on grids.
%
This induces substantial difficulties in the proof of the
sharp rate compared to the regular quantization setting.
\end{remark}

\begin{prop}[Scaling property]\label{prop:scaleProp}
Let $C= a+\rho [0,1]^d$ ($a\!\in \R^d$, $\rho>0$) be a $d$-dimensional hypercube, with edges parallel to the coordinate axis and edge-length $\rho > 0$. Then
\[
	d_{n,p}(\U(C)) = \rho \cdot d_{n,p}(\Unifd).
\]
\end{prop}
\noindent {\it Proof.} Keeping in mind that $\lambda_d([0,\rho]^d)=\rho^d$, it
holds that
\begin{eqnarray*}
\begin{split}
\hskip 1,5 cm  \dqp(\U(C); \{a+\rho x_1, \ldots, a+ \rho x_n\}) & = \int_{[0, \rho]^d} \underset{
  \text{s.t. } \left[  \begin{smallmatrix} \rho x_1 & \ldots & \rho x_n\\
          1 & \ldots & 1\\
        \end{smallmatrix}  \right] \lambda =
      \left[\begin{smallmatrix}
         \xi\\ 1\\
        \end{smallmatrix} \right],\,
      \lambda \geq 0 }{\min_{\lambda\in \R^n }\sum_{i=1}^n \lambda_i
      \,
  \norm{\xi-\rho x_i}^p}
  \frac{\lambda_d(d\xi)}{\lambda_d\bigl([0,\rho]^d\bigr)}\\ & =  \int_{[0,
  1]^d} \underset{ \text{s.t. } \left[  \begin{smallmatrix} \rho x_1 & \ldots &
  \rho x_n\\ 1 & \ldots & 1\\
        \end{smallmatrix}  \right] \lambda =
      \left[\begin{smallmatrix}
         \rho u\\ 1\\
        \end{smallmatrix} \right],\,
      \lambda \geq 0 }{\min_{\lambda\in \R^n }\sum_{i=1}^n \lambda_i
      \,
  \norm{\rho u-\rho x_i}^p} \lambda_d(du)\\
  & = \rho^p \int_{[0, 1]^d} \LP{u} \lambda_d(du)\\
  & = \rho^p \cdot \dqp(\Unifd; \{x_1, \ldots, x_n\}).\hskip 3 cm \Box
\end{split} 
\end{eqnarray*}

The following lemma shows that also for $\bar d_{n,p}$ the convex hull 
spanned by a sequence of ``semi-optimal'' quantizers asymptotically covers the
interior of $\supp(\ProbX)$. This fact is trivial for $d_{n,p}$ if $X$ has  a compact
support.  

\begin{lemma}\label{lem:radius} Let $K  = \conv\{a_1, \ldots, a_k\} \subset  \opSuppP$ be a set with $\mathring
K \neq \emptyset$ and let $\Gamman$ be a sequence of quantizers such that $\bar
d_{n,p}(\Probb,\Gamman)\to 0$ as $n\to +\infty$. Then there exists $n_0\in\N$ such that for all $n\geq n_0$
\[
K \subset \conv(\Gamman).
\]
\end{lemma}

\noindent {\it Proof.} Set $a_0 = \frac{1}{k} \sum_{i=1}^k a_i$ and define for every $\rho > 0$
\[
\tilde K(\rho) = \conv\{\tilde a_1(\rho), \ldots, \tilde a_k(\rho)\}
\quad \text{with}\quad
\tilde a_i(\rho)= a_0 + (1+\rho) (a_i - a_0).
\]

Since $K\subset\opSuppP$ there exists  $\rho_0 > 0$ such that $\tilde K =
\tilde K(\rho_0) \subset \supp(\Probb)$. From now on, we denote $\tilde
a_i(\rho_0)$ by $\tilde a_i$.
Since moreover $\tilde a_i \!\in \supp(\Probb)$, there exists a sequence
$(a^{n}_i)_{n\ge 1}$ having values in $\conv(\Gamman)$ and converging to
$\tilde a_i$. Otherwise there would exist $\varepsilon_0>0$ and a subsequence
$(n')$ such that $B(\tilde a_i,\varepsilon_0)\subset (\conv(\Gamma_{\!n'}))^c$.
Then 
$$
\bar d_{n'}^p(X,\Gamma_{\!n'}) \ge \E\, {\rm dist}(X,\Gamma_{\!n'})^p\mbox{\bf 1}_{\{X\in B(\tilde a_i,\varepsilon_0/2)\}}
\left(\frac{\varepsilon_0}{2}\right)^p \Probb(B(\tilde a_i,\varepsilon_0/2))>0
$$
since $\tilde a_i\!\in {\rm supp}(\Probb)$. This contradicts the assumption on the sequence
$(\Gamman)_{n\ge 1}$.  

Since $K$ has a nonempty interior, it follows that $\adim
\{a_1, \ldots, a_k\} = \adim \{\tilde a_1, \ldots, \tilde a_k\} = d$. Consequently,
we may choose a subset $I^\ast \subset\{1, \ldots, k\}, \,\abs{I^\ast} = d+1$, so
that $\{\tilde a_j: j\in I^\ast\}$ is an affinely independent system in $\R^d$
and furthermore there exists  $n_0 \in \N$ such that the same holds for
$\{a^n_j: j\in I^\ast\}$, $n \geq n_0$.
Hence, we may write for $n\geq n_0$
\begin{equation}\label{eq:lemmaoneproofLS}
 \tilde a_i= \sum_{j\in I^\ast} \mu^{n,i}_ja^n_j,\quad
 \sum_{j\in I^\ast}\mu^{n,i}_j=1,\quad i=1,\ldots,k. 
 \end{equation}
 This linear system has the unique asymptotic solution $\mu^{\infty,i}_j =
 \delta_{ij}$ (Kronecker symbol), which implies $\mu^{n,i}_j\to \delta_{ij}$
when  $n\to+\infty$.

Now let $\xi\in K\subset \tilde K$ and write
\[
	\xi = \sum_{i=1}^k \lambda_i a_i\;
	\text{ for some }\; \lambda_i \geq 0, \sum_{i=1}^k \lambda_i = 1.
\]
One easily checks that it also holds 
\[
	\xi = \sum_{i=1}^k \tilde \lambda_i \tilde a_i\, \text{ with  } \, \tilde
	\lambda_i = \frac{\rho_0}{k(1+\rho_0)} + \frac{\lambda_i}{ 1+\rho_0} \geq \frac{\rho_0}{k(1+\rho_0)} > 0 \quad  \text{and}\quad \sum_{i=1}^k
	\tilde \lambda_i = 1.
\]

Furthermore, we may choose  $n_1 \geq n_0$ such that, for every $n\geq n_1$,
\[
\mu^{n,i}_i > \frac{1}{2} \quad \text{ and } \quad  \forall j \neq i,\; \abs{\mu^{n,i}_j} \leq
\frac{\rho_0}{4k(1+\rho_0)}.
\]

Using (\ref{eq:lemmaoneproofLS}),  this leads to
\[
 \xi = \sum_{j\in I^\ast} \Bigl(\sum_{i=1}^k \tilde \lambda_i  \mu^{n,i}_j
 \Bigr) a^n_j
\]
and 
\[
	\sum_{i=1}^k \tilde \lambda_i  \mu^{n,i}_j > \tilde \lambda_j  \mu^{n,j}_j -
	\sum_{i=1,i\neq j}^k \tilde \lambda_i  \abs{\mu^{n,i}_j} >
	\frac{\rho_0}{2k(1+\rho_0)} - \frac{\rho_0}{4k(1+\rho_0)} =
	\frac{\rho_0}{4k(1+\rho_0)} > 0, \;j \in I^\ast.
\]
Finally, one completes the proof by noting that 
$\displaystyle 	\sum_{j\in I^\ast} \sum_{i=1}^k \tilde \lambda_i  \mu^{n,i}_j = \sum_{i=1}^k \tilde
	\lambda_i \sum_{j\in I^\ast} \mu^{n,i}_j = 1.  \quad \Box$

As already said, Proposition~\ref{prop:subLin} does not hold anymore for $\bar
d_{n,p}$. As a consequence we have to establish a 
``firewall Lemma",
which will be a useful tool to overcome this problem  in the
non-compact setting.

\begin{lemma}[Firewall]\label{lem:firewall} Let $K \subset \R^d$ be compact and convex with $\mathring K \neq \emptyset$.
Moreover, let $\varepsilon > 0$ be small enough so that 
\[
	K_\varepsilon = \{ x\in K : \dist_{\ell^{\infty}}(x, K^c) \geq \varepsilon \} \neq \emptyset.
\]
Let $\Gae$ be a subset of the lattice $\alpha \Z^d$ with edge-length $\alpha > 0$
satisfying
\[
 K \setminus K_\varepsilon \subset \conv(\Gae)\;\mbox{ and }\; \forall\, x\!\in K\setminus K_{\varepsilon},\;{\rm
dist}_{\norm{\cdot}}(x,\Gamma_{\alpha,\varepsilon})\le C_{\norm{\cdot}}\alpha
\] 
where $C_{\norm{\cdot}} > 0$ is  a real  constant   only  depending on the norm $\norm{\cdot}$.

Then, for every grid $\Gamma\subset \R^d$ containing $K$ and every $\eta \in (0,1)$, it holds
\[
\forall\, \xi \in
K_\varepsilon,\quad 	\Fp(\xi; \Gamma) \geq \frac{1}{(1+\eta)^{p+d+1}} \Fp(\xi; (\Gamma \cap \mathring K)
	\cup \Gae )   - (1+\eta)^{-d-1}\eta^{-p} 
	(d+1)\, C^p_{\norm{\cdot}}  \alpha^p.
\]
\end{lemma}

\begin{remark} The lattice $\Gamma_{\alpha,\varepsilon}$ and its size will be carefully defined and estimated  for the specified compact sets $K$ when calling upon the firewall lemma  in what follows.
\end{remark}

{\sc Proof.}
Let $\Gamma = \{ x_1, \ldots, x_n\}$ and let $\xi\!\in K_\varepsilon$.
Then we may choose $I=I(\xi)\subset \{1, \ldots, n\}$, $|I|\le d+1$ such that
\[
\Fp(\xi; \Gamma) = \sum_{i\in I} \lambda_j \norm{\xi - x_i}^p, \quad
\sum_{i\in I} \lambda_i x_i= \xi,\, \lambda_i \geq 0,\, \sum_{i\in I} \lambda_i
= 1.
\]
If for every $x_i \!\in \Gamma \setminus \mathring K$ $\lambda_i=0$ then $F^p(\xi,\Gamma)= F^p(\Gamma\cap \mathring K)$ and our claim is trivial.
Therefore, let $J(\xi)=\{i\,:\, x_i \!\in \Gamma \setminus \mathring K, \,
\lambda_i>0\}\subset I(\xi)$ and choose one fixed  $i_0 \in J(\xi)$.
 Let  $\theta= \theta(i_0)\! \in (0,1)$ such that 
 \[ \tilde x_{i_0} = \xi +
\theta (x_{i_0} - \xi) \in K \setminus K_\varepsilon \quad\mbox{and}\quad \frac{\theta^{p\wedge 1}}{\theta+\lambda_{i_0}(1-\theta)} \le 1+\eta
\]
(when $p\ge 1$ the   right constraint is
empty).
Setting
\[
	\tilde \lambda^0_i = \frac{\lambda_i \theta}{\theta + \lambda_{i_0}(1-\theta)},\,
	i \in I\setminus\{i_0\}, \quad \tilde \lambda^0_{i_0} =
	\frac{\lambda_{i_0}}{\theta + \lambda_{i_0}(1-\theta)}
\]
we arrive at
\[ 
\tilde \lambda^0_{i_0} \tilde x_{i_0} +\sum_{i\in I \setminus \{i_0\}} \tilde
\lambda^0_i x_i = \xi, \;\tilde \lambda^0_i \geq 0, \;\sum_{i\in I}
\tilde \lambda^0_i = 1.
\] 
Consequently 
\begin{equation*}
\begin{split}
\tilde \lambda^0_{i_0} \norm{\xi - \tilde x_{i_0}}^p + \sum_{j\in I \setminus
\{i_0\}} \tilde \lambda^0_i \norm{\xi - x_i}^p  &=
\frac{\lambda_{i_0}\theta^p}{\theta + \lambda_{i_0}(1-\theta)} \norm{\xi -
x_{i_0}}^p
+ \sum_{i\in I \setminus
		\{i_0\}}\frac{\lambda_i \theta}{\theta + \lambda_{i_0}(1-\theta)}\norm{\xi - x_i}^p \\
& \leq \frac{\theta^{p\wedge 1}}{\theta + \lambda_{i_0}(1-\theta)} \sum_{i\in I}\lambda_i
		\norm{\xi - x_i}^p\\
&		\le(1+\eta)\sum_{i\in I}\lambda_i
		\norm{\xi - x_i}^p.
\end{split}
\end{equation*}
Repeating the
procedure 
for every $i\!\in J(\xi)$ 
finally yields by induction the existence of $\tilde{x}_i\!\in K\setminus
K_{\varepsilon}$ and  $\tilde \lambda_i$, $i\in I$ such that
 \[ \sum_{i\in I :
x_i\notin \mathring{K}} \tilde \lambda_{i} \tilde x_{i} +\sum_{i\in I : x_i\in
\mathring{K}} \tilde \lambda_i x_i = \xi, \;\tilde \lambda_i \geq 0,
\;\sum_{i\in I} \tilde  \lambda_i = 1 
\] 
and 
\begin{equation}\label{eq:FirewallProofIneq}
(1+\eta)^{|J(\xi)|} \Fp(\xi; \Gamma)\ge \sum_{i\in I :
x_i\notin \mathring{K}} \tilde \lambda_{i} \norm{\xi - \tilde x_{i}}^p
+\sum_{i\in I : x_i\in \mathring{K}} \tilde \lambda_i \norm{\xi -  x_{i}}^p. 
\end{equation}

Let us denote $\Gae = \{a_1, \ldots, a_m\}$ and let   $i_0\!\in J(\xi)$ so that $\tilde x_{i_0}$ is a
``modified" $x_{i_0}$ (originally lying in $\Gamma\setminus \mathring K$). By
construction $\tilde x_{i_0}\in K \setminus K_\varepsilon\subset\conv(\Gae)$ and there is
$J_{i_0}\subset \{1, \ldots, m\}$ such that 
\[ \Fp(\tilde x_{i_0}, \Gae) =
\sum_{j\in J_{i_0}} \mu^{i_0}_j \norm{\tilde x_{i_0} - a_j}^p, \; \sum_{j\in
J_{i_0}} \mu^{i_0}_j x_j = \tilde x_{i_0},\, \mu^{i_0}_j \geq 0,\, \sum_{j\in
J_{i_0}} \mu^{i_0}_j = 1 
\] 
and 
\[ \forall\, j\!\in J_{i_0},\quad \norm{\tilde
x_{i_0}-a_j}\le C_{\norm{\cdot}}\,\alpha. 
\] 


 \smallskip Using the elementary inequality
\[ 
\forall\,p>0, \; \forall \eta > 0, \;\forall\, u,v \geq 0,\quad (u+v)^p \leq (1+\eta)^p u^p +
\Bigl(1+\frac{1}{\eta} \Bigr)^p v^p, 
\] 
we derive that for every $j\!\in J_{i_0}$
\begin{equation*}
  \norm{\xi - a_j}^p  \leq \bigl( \norm{\xi - \tilde x_{i_0}} + \norm{\tilde
  x_{i_0} - a_j}  \bigr)^p  \leq (1+\eta)^p \norm{\xi - \tilde x_{i_0}}^p + \Bigl(1+\frac{1}{\eta}
  \Bigr)^p\, C^p_{\norm{\cdot}}\, \alpha^p.
\end{equation*}

As a consequence, 
\[ 
\sum_{j\in J_{i_0}} \mu^{i_0}_j \norm{\xi - a_j}^p  \leq
(1+\eta)^p \norm{\xi - \tilde x_{i_0}}^p + \Bigl(1+\frac{1}{\eta}
  \Bigr)^p\, C^p_{\norm{\cdot}}\, \alpha^p
\] 
which in turn implies
 \[
\norm{\xi - \tilde x_{i_0}}^p\ge \frac{1}{(1+\eta)^p}\sum_{j\in J_{i_0}} \mu^{i_0}_j \norm{\xi - a_j}^p- \eta^{-p}\, C^p_{\norm{\cdot}}\, \alpha^p.
\] 
Plugging this inequality in~(\ref{eq:FirewallProofIneq}) yields and using that
$|J(\xi)|\le d+1$, we finally get
\begin{equation*}
\begin{split}
 (1+\eta)^{|J(\xi)|} \Fp(\xi; \Gamma) 
  \ge &  \sum_{i\in I : x_i\in \mathring{K}} \tilde \lambda_i  \norm{\xi -  x_{i}}^p  +  \frac{1}{(1+\eta)^p}\sum_{i\in I : x_i\notin \mathring{K}}\tilde \lambda_i\sum_{j\in J_{i}}\mu^i_j\norm{\xi-a_j}^p\\
		&- |J(\xi)|\eta^{-p} d\, C^p_{\norm{\cdot}}\, \alpha^p\\
   \geq  &\; \frac{1}{(1+\eta)^p} \Fp\bigl(\xi; (\Gamma \cap \mathring K\})
  \cup \Gae \bigr) -\eta^{-p}\,(d+1)\,   C^p_{\norm{\cdot}}\, \alpha^p.\hfill \Box
\end{split}
\end{equation*}

Now we can establish
the sharp rate for the uniform distribution $U([0,1]^d)$.

\begin{prop}[Uniform distribution]\label{prop:rateU} For every $p\ge 1$,
\[
\Qpn:= \inf_{n\geq 0} n^{1/d}\,
d_{n,p}\bigl(\Unifd \bigr) = \limn n^{1/d}\, d_{n,p}\bigl(\Unifd \bigr).
\]
\end{prop}
{\sc Proof.}
Let $n,m \in \N,\, m<n$ and set $k = k(n,m) = \left\lfloor
\bigl(\frac{n}{m}\bigr)^{1/d} \right\rfloor \ge 1$.

Covering the unit hypercube $[0,1]^d$ by $k^d$ translates $C_1, \ldots,
C_{k^d}$ of the hypercube $\bigl[0,\frac{1}{k}\bigr]^d$,
we arrive at $\Unifd = k^{-d} \sum_{i=1}^{k^d} \U(C_i)$. Hence, Proposition~\ref{prop:subLin} yields
\[
d_{n,p}^p\bigl(\Unifd \bigr) \leq k^{-d}\sum_{i=1}^{k^d} \dqp_m(\U(C_i)).
\]
Furthermore, Proposition~\ref{prop:scaleProp} implies
\[
d_{m,p}(\U(C_i)) = k^{-1}\, d_{m,p}\bigl(\Unifd \bigr),
\]
so that we may conclude for all $n,m\in\N,\, m<n$,
\[
	d_{n,p}\bigl(\Unifd \bigr) \leq k^{-1}\, d_{m,p}\bigl(\Unifd \bigr).
\]
Thus, we get 
\begin{eqnarray*}
\begin{split}
  n^{1/d}\,d_{n,p}\bigl(\Unifd \bigr) & \leq k^{-1}\, n^{1/d}\,d_{m,p}\bigl(\Unifd
  \bigr)\\
  	& \leq  \frac{k+1}{k} \, m^{1/d}\, d_{m,p}\bigl(\Unifd
  \bigr), 
\end{split}
\end{eqnarray*}
which yields for every fixed integer $m \geq 1$
\[
	\limsn n^{1/d}\, d_{n,p}\bigl(\Unifd \bigr) \leq m^{1/d}\, d_{m,p}\bigl(\Unifd
	\bigr),
\]
since $\limn k(n,m) = +\infty$. This finally implies
\[
\limn n^{1/d}\,d_{n,p}\bigl(\Unifd \bigr)  = \inf_{m\geq 0} m^{1/d}\,
d_{m,p}\bigl(\Unifd \bigr).\hskip 3 cm  \Box
\]

\begin{prop}\label{prop:rateUNN} For every $p\ge 1$, 
\[ 
\Qpn = \limn n^{1/d}\, d_{n,p}\bigl(\Unifd \bigr)  = \limn n^{1/d}\,
\bar d_{n,p}\bigl(\Unifd \bigr) 
\]
\end{prop}
\begin{proof}
Since for every compactly supported distribution $\Probb$ we have $\bar d_{n,p}(\Probb) \leq  d_{n,p}(\Probb)$
it remains to show
\[
	\Qpn \leq \limin n^{1/d}\, \bar d_{n,p}\bigl(\Unifd \bigr).
\]


For $0 < \varepsilon < 1/2$ let $C_\varepsilon= (1/2,\ldots,1/2)+
\frac{1-\varepsilon}{2}[-1,1]^d$ be the centered hypercube in $[0,1]^d$ with
edge-length $1-\varepsilon$ and midpoint $(1/2,\ldots,1/2)$.
Moreover let $(\Gamman)$ be a sequence of quantizers such that, for every $n\ge 1$, 
\[
\bar d_{p}(\Unifd; \Gamman) \leq (1+\varepsilon) \bar d_{n,p}(\Unifd).
\]

Owing to Lemma~\ref{lem:radius}, as $C_{\varepsilon} \subset (1,1)^d$, there is an integer $n_\varepsilon \in \N$ such
that \[ \forall n \geq n_\varepsilon, \quad C_\varepsilon \subset
\conv(\Gamman) . 
\]

We therefore get for any $n\geq n_\varepsilon$
 \begin{eqnarray*}
 (1+\varepsilon)^d\bar d^p_{n}\bigl(\Unifd \bigr)& \geq &\bar d^p\bigl(\Unifd;\Gamman \bigr)\\
 &\ge & \int _{C_{\varepsilon}}\bar F^p(\xi,\Gamma_n)^pd\xi=\int _{C_{\varepsilon}} F^p(\xi,\Gamma_n)^pd\xi= \lambda_d(C_{\varepsilon}) d^p \bigl(\U(C_\varepsilon), \Gamma_n\bigr) \\
&\ge& (1 - \varepsilon)^{d} d^p_{n}\bigl(\U(C_\varepsilon)\bigr)  = (1 - \varepsilon)^{d+p} d^p_{n}\bigl(\Unifd \bigr) 
\end{eqnarray*}
where we used the scaling property (Proposition~\ref{prop:scaleProp}) in the last line.

Hence, we obtain for all $0 < \varepsilon < 1/2$
\[
\limin n^{1/d}\, \bar d_{n,p}\bigl(\Unifd \bigr) \geq \frac{(1 -
\varepsilon)^{1+d/p}}{(1+\varepsilon)^{d/p}}\, \Qpn,
\]
so that letting $\varepsilon \to 0$ completes the proof. 
\end{proof}

\begin{prop}\label{prop:rateCubewise}
Let $\Probb = \sum_{i=1}^m s_i\, \U(C_i),\, \sum_{i=1}^m s_i = 1$,  $s_i>0$, $i=1,\ldots,m$,
where $C_i=a_i+[0,l]^d$,  $i=1,\ldots,m$, are pairwise disjoint hypercubes in $\R^d$ with common edge-length
$l$. Set 
\[
h := \frac{d\Probb}{d\lambda_d}= \sum_{i=1}^m s_i
l^{-d}\ind{C_i}.
\]
Then
\[
 \limn n^{1/d}\, d_{n,p}(\Probb)=  \limn n^{1/d}\, \bar d_{n,p}(\Probb)= \Qpn\cdot \normdp{h}^{\frac 1p} .
\]
\end{prop}
\begin{proof}Since $d_{n,p}(\Probb)\ge \bar d_{n,p}(\Probb)$ it suffices to show that 
\[
 \limsn n^{1/d}\, d_{n,p}(\Probb) \leq \Qpn\cdot \normdp{h}^{\frac 1p} \quad 
 \mbox{ and } \quad \limin n^{1/d}\, \bar d_{n,p}(\Probb) \geq \Qpn\cdot \normdp{h}^{\frac 1p} .
\]
For $n\in\N$, set
\[
t_i = \frac{s_i^{d/(d+p)}}{\sum_{j=1}^m s_j^{d/(d+p)}} \qquad\text{ and }\qquad
n_i = \lfloor t_i n\rfloor, \,1\leq i \leq m.
\]

Then, by Proposition~\ref{prop:subLin} and Proposition~\ref{prop:scaleProp}, we
get for every $n\geq \max_{1\leq i \leq m}(1/t_i)$
\[
d_{n}^p(\Probb) \leq \sum_{i=1}^m s_i\, d_{n}^p(\U(C_i)) = l^p \sum_{i=1}^m s_i\,
d^p_{n_i}(\Unifd).
\]

Proposition~\ref{prop:rateU} then yields
\[
n^{\frac pd}\, \dqp_{n_i}(\Unifd) = \biggl(\frac{n}{n_i}\biggr)^{\frac pd}\,
n_i^{\frac pd}\,\dqp_{n_i}(\Unifd) \longrightarrow  t_i^{-\frac pd} \Qpn \quad\text{ as } n\to+\infty.
\]
Noting that  $\normdp{h} = l^p \Bigl( \sum s_i^{d/(d+p)} 
\Bigr)^{(d+p)/d}$,
we get
\[
\limsn n^{\frac pd}\, \dqpn(\Probb) \leq \Qpn  l^p \sum_{i=1}^m s_i\, t_i^{-\frac pd} = \Qpn\cdot
\normdp{h}.
\]

\noindent $(b)$ Let $\varepsilon \in(0,l/2)$ and let $\Cie$ denote  the closed hypercube with
the same center as $C_i$ but with edge-length $l-\varepsilon$.
For $\alpha \in (0, \varepsilon/2)$, we set $\tilde \alpha = \frac{l}{\lceil l/\alpha\rceil }$ and we define the lattice 
$$
 \Gamma_{\alpha, \varepsilon,i} = \big(a_i+ \tilde \alpha\Z^d\big)\cap (C_i\setminus C_{i,\varepsilon}\big) \bigcup \{\mbox{vertices of }  C_i\}.
$$

It is clear that $\conv(\Gamma_{\alpha,\varepsilon,i}) = C_i \subset C_i\setminus C_{i,\varepsilon}$ since it contains the vertices of $C_i$.
Moreover, for every $\xi\!\in C_i\setminus C_{i,\varepsilon}$, ${\rm dist}_{\ell^{\infty}}(\xi,\Gamma_{\alpha,\varepsilon,i})\le \alpha$ so that there exists a real constant $C_{\norm{\cdot}}>0$ only depending on the norm $\|.\|$ such that ${\rm dist}_{\|.\|}(\xi,\Gamma_{\alpha,\varepsilon,i})\le C_{\norm{\cdot}} \alpha$. Consequently the lattice $\Gamma_{\alpha,\varepsilon,i}$ satisfies the assumption of the firewall lemma (Lemma~\ref{lem:firewall}).
 
 On the other hand, easy combinatorial arguments show that number of points $m_i$ of $\Gamma_{\alpha,\varepsilon,i}$ falling in $C_i$ satisfies $\lceil \frac{l}{\tilde \alpha}\rceil^d \le m_i \le \big(\lceil \frac{l}{\tilde \alpha}\rceil+1\big)^d +2^d$ whereas the number $m_{i,\varepsilon}$ of points falling in $C_{i,\varepsilon}$ satisfies  $\big(\lceil \frac{l-\varepsilon}{\tilde \alpha}\rceil-1\big)^d \le m_{i,\varepsilon} \le \big(\lceil \frac{l-\varepsilon}{\tilde \alpha}\rceil+1\big)^d$ so that
 \[
\Big\lceil\frac{l}{\tilde \alpha}\Big\rceil^d - \Big(\Big\lceil\frac{l-\varepsilon}{\tilde \alpha}\Big\rceil+1\Big)^d \le  |\Gamma_{\alpha,\varepsilon,i}| \le \Big(\Big\lceil\frac{l}{\tilde \alpha}\Big\rceil+1\Big)^d +2^d-\Big(\Big\lceil\frac{l-\varepsilon}{\tilde \alpha}\Big\rceil-1\Big)^d.
 \]
%
%
We define for every $\varepsilon\in(0,l/2), \alpha\in(0,\varepsilon/2)$
\[
	g_{l,\varepsilon}(\alpha) = \alpha^d \abs{\Gamma_{\alpha,\varepsilon,i}}.
\]
Since $\frac{\alpha}{\tilde{\alpha}} \to 1$ and $2\alpha \Bigl\lceil
	\frac{\varepsilon/2}{\tilde{\alpha}} \Bigr\rceil \to \varepsilon$ as $\alpha
	\to 0$, we conclude from the above inequalities that
	\begin{equation}\label{eq:proofFirewallDefgl}
	\forall \varepsilon \in (0, l/2),\quad	\lim_{\alpha \to 0} g_{l,\varepsilon}(\alpha) = l^d - (l-\varepsilon)^d.	 
	\end{equation}	
Let $\eta\in(0,1)$ and denote by $\Gamman$ a sequence of
$n$-quantizers such that $\bar d^p(\Probb; \Gamman) \leq (1+\eta)
d^p_n(\Probb)$. It follows from Proposition~\ref{PdtQErrop} that
$\bar d^p(\Probb; \Gamman) \to 0$ for $n\to \infty$ so that 
Lemma~\ref{lem:radius} yields the existence of $n_\varepsilon\in\N$ such that for any $n\geq n_\varepsilon$
\[
\bigcup_{1\leq i \leq m} \Cie \subset \conv(\Gamman).
\]

We then derive from Lemma~\ref{lem:firewall} (firewall) 
\begin{equation*}
\begin{split}
  \dqbp(\U(C_i); \Gamman) & = l^{-d} \int_{C_i} \Fbp(\xi; \Gamman)\,
  \lambda_d(d\xi) \\
  & \geq l^{-d} \int_{\Cie} \Fbp(\xi; \Gamman)\,
  \lambda_d(d\xi) = l^{-d} \int_{\Cie} \Fp(\xi; \Gamman)\,
  \lambda_d(d\xi) \\
  & \geq \frac{l^{-d }\,(l-\varepsilon)^d}{(1+\eta)^{p+d+1}}\, \dqp\bigl(\U(\Cie);
  (\Gamma_n\cap\mathring C_i)\cup \Gamma_{\alpha,\varepsilon,i}\bigr) -
  l^{-d}\,(l-\varepsilon)^d \frac{(1+\eta)^{-d-1}}{\eta^{p}}\, (d+1) C_{\norm{\cdot}}\cdot\alpha^p.
\end{split}
\end{equation*}

At this stage, we set for every $\rho > 0$
\begin{equation}\label{eq:proofFirewallDefalpha}
\alpha_n=\alpha_n(\rho) = \Bigl( \frac{m}{\rho n} \Bigr)^{1/d}
\end{equation}
and denote
\[
	n_i = \abs{(\Gamman\cap\mathring C_i)\cup \gaei}.
\]
Proposition~\ref{prop:scaleProp} yields $d_{n_i,p}(\U(\Cie)) =
(l-\varepsilon) d_{n_i,p}(\Unifd) $, so that we get 
\begin{equation}\label{eq:proofRateCubewiseb}
  \begin{split}
      n^{\frac pd}d_{n}^p(\Probb) & \geq \frac{1}{1+\eta} \sum_{i=1}^m s_i\,
      n^{\frac pd}\,\dqbp(\U(C_i);\Gamman)\\
      & \geq \frac{l^{-d }\,(l-\varepsilon)^d}{(1+\eta)^{p+d+2}}\, \sum_{i=1}^m s_i\, n^{\frac pd}\, \dqp\bigl(\U(\Cie);
  (\Gamma_n\cap\mathring C_i)\cup \gaei\bigr) \\
  & \qquad -  l^{-d}\,(l-\varepsilon)^d
  \frac{(1+\eta)^{-d-2}}{\eta^p} \sum_{i=1}^m
  s_i\, (d+1)\, C_{\norm{\cdot}}\cdot\alpha^p \cdot
  n^{\frac pd} \\
  &\geq \frac{l^{-d }\,(l-\varepsilon)^{d+p}}{(1+\eta)^{p+d+2}}\, \sum_{i=1}^m s_i\,
      n^{\frac pd}\, \dqpnn{n_i}\bigl(\Unifd\bigr) 
  -  l^{-d}\,(l-\varepsilon)^d
  \frac{(1+\eta)^{-d-2}}{\eta^p}  (d+1)\,
       C_{\norm{\cdot}} 
  \Bigl(\frac{m}{\rho}\Bigr)^{\frac pd}.
  \end{split}
\end{equation}

Since
\[
\frac{n_i}{n} \leq \frac{\abs{\Gamman\cap\mathring C_i}}{n} +
\frac{g_{l,\varepsilon}(\alpha_n)}{n\alpha_n^d} = \frac{\abs{\Gamman\cap\mathring
C_i}}{n} + \frac{\rho}{m} g_{l,\varepsilon}(\alpha_n),
\]
we conclude from (\ref{eq:proofFirewallDefgl}) and (\ref{eq:proofFirewallDefalpha}) that
\[
	\limsn\sum_{i=1}^m \frac{n_i}{n} \leq 1 + \rho \bigl(l^d -
	(l-\varepsilon)^d\bigr).
\]
We may choose a subsequence (still denoted by $(n)$), such that
\[
n^{1/d}\, \bar
d_{n,p}(\Probb) \to \limin n^{1/d}\, d_{n,p}(\Probb)
 \qquad\text{ and }\qquad 
\frac{n_i}{n} \to v_i \in [0,1 + \rho (l^d -
	(l-\varepsilon)^d)].
\]

As a matter of fact, $v_i>0,$ for every $i=1,\ldots  m$: otherwise
Proposition~\ref{prop:rateU} would yield
\begin{equation*}
\begin{split}
n^{\frac pd}\, \dqbpn(\Probb)  \geq& \frac{l^{-d}\,(l-\varepsilon)^{d+p}}{(1+\eta)^{p+d+2}}
\sum_{i=1}^m s_i\, \Bigl(\frac{n_i}{n}\Bigr)^{-\frac pd} n_i^{\frac pd}\,
\dqpnn{n_i}\bigl(\Unifd\bigr)\\
& \;-\,  l^{-d}\,(l-\varepsilon)^d
  \frac{(1+\eta)^{p-d-2}}{\eta^p} (d+1) C_{\norm{\cdot}}\cdot
  \Bigl(\frac{m}{\rho}\Bigr)^{\frac pd}\\
  & \to +\infty \quad \mbox{ as } n\to +\infty
  \end{split}
\end{equation*}
which contradicts $(a)$. Consequently, we may normalize the $v_i$'s by setting
\[
	\widetilde v_i = \frac{v_i}{1+\rho (l^d - (l-\varepsilon)^d)}, \; i=1,\ldots,m,
\]
so that $\sum_{i=1}^m \widetilde v_i \leq 1$. We derive from 
Proposition~\ref{prop:rateU} that
\begin{equation*}
\begin{split}
  \limin \sum_{i=1}^m s_i\, n^{\frac pd}\,
  \dqpnn{n_i}\bigl(\Unifd\bigr) 
  &\ge\sum_{i=1}^m s_i\,v_i^{-\frac pd}   n_i^{\frac pd}\,
  \dqpnn{n_i}\bigl(\Unifd\bigr) \\ 
  & = \Qpn (1+\rho (l^d - (l-\varepsilon)^d)^{-\frac pd}\sum_{i=1}^m s_i\, \widetilde v_i^{\,-\frac pd}\\
 &\geq \Qpn(1+\rho (l^d - (l-\varepsilon)^d)^{-\frac pd}\inf_{\sum_i y_i\le 1, y_i\ge 0} \sum_{i=1}^m s_i y_i^{-\frac pd}
\\
 &= \Qpn(1+\rho (l^d - (l-\varepsilon)^d)^{-\frac pd} \biggl(\sum_{i=1}^m s_i^{d/(d+p)} \biggr)^{(d+p)/d}.
  \end{split}
\end{equation*}

Hence, we derive from 
(\ref{eq:proofRateCubewiseb})
\begin{equation*}
\begin{split}
  \limin n^{\frac pd}\, \dqbpn(\Probb) & 
  \geq \frac{l^{-d }\,(l-\varepsilon)^{d+p}}{(1+\eta)^{p+d+2}\bigl(1+\rho (l^d -
  (l-\varepsilon)^d)\bigr)^{\frac pd}}\,\Qpn\,\biggl(\sum_{i=1}^m s_i^{d/(d+p)}
  \biggr)^{(d+p)/d}
  \\ 
 & \qquad -  l^{-d}\,(l-\varepsilon)^d
  \frac{(1+\eta)^{-d-2}}{\eta^p} (d+1)\,        C_{\norm{\cdot}}\cdot
  \Bigl(\frac{m}{\rho}\Bigr)^{\frac pd}. \\ 
\end{split}
\end{equation*}
Letting $\varepsilon\to 0$ implies
\begin{equation*}
\begin{split}
  \limin n^{\frac pd}\, \dqbpn(\Probb) & \geq
  \frac{l^{p}}{(1+\eta)^{p+d+2}}\, \Qpn \,\biggl(\sum_{i=1}^m s_i^{d/(d+p)}
  \biggr)^{(d+p)/d}
 \, - \,  \frac{(1+\eta)^{-d-2}}{\eta^p}(d+1)\,  
       C_{\norm{\cdot}} 
  \Bigl(\frac{m}{\rho}\Bigr)^{\frac pd}\\
  & =  \frac{1}{(1+\eta)^{p+d+2}}\,\Qpn \cdot\normdp{h} 
 \, - \, \frac{(1+\eta)^{-d-2}}{\eta^p}d\,
       C_{\norm{\cdot}} 
  \Bigl(\frac{m}{\rho}\Bigr)^{\frac pd}
\end{split}
\end{equation*}
and, finally, letting successively  $\rho$ go to $+\infty$ and $\eta$ go to $0$
completes the proof.
\end{proof}

\begin{prop}\label{prop:rateCompact}
Assume that $\Probb$
is  absolutely continuous w.r.t. $\lambda_d$ with  compact support. Then
\[
 \limn n^{\frac pd}\,d_{n,p}(\Probb) = \limin n^{\frac pd}\, \bar d_{n,p}(\Probb)= \Qpn\cdot \normdp{h}^{\frac 1p}
\]
\end{prop}
 {\sc Proof.} Since $d_{n,p}(\Probb)\ge \bar d_{n,p}(\Probb)$ it suffices to show that 
\[ \limsn n^{\frac pd}\,d_{n,p}(\Probb) \leq \Qpn\cdot \normdp{h}^{\frac 1p}\; \mbox{ and }\;\displaystyle   \limin n^{\frac pd}\, \bar d_{n,p}(\Probb) \geq \Qpn\cdot \normdp{h} ^{\frac1p}.
\]
\noindent {\em Preliminary step.} Let $C=[-l/2,l/2]^d$ be a  closed hyper hypercube centered at the origin, parallel to the coordinate axis  with
edge-length $l$, such that $\supp(\Probb) \subset C$.
For $k\in\N$ consider the tessellation of $C$ into $k^d$ closed hypercubes with
common edge-length $l/k$.
To be precise, for every $\underline i=(i_1,\ldots,i_d)\!\in \Z^d$, we set
\[
C_{\underline i} =\Prod_{r=1}^d\Big [-\frac l2+\frac{i_rl}{k},-\frac l2+\frac{(i_r+1)l}{k}\Big ].
\]
Then, set \begin{equation} 
h = \frac{d\Probb}{d\lambda_d}\;\mbox{ and }\; 
  \Probb_k  = \sum_{\substack{\underline i\in \Z^d\\ 0\leq i_r < k}}
  \Probb(C_{\underline i}) \, \U(C_{\underline i}),\;h_k =
  \frac{d\Probb_k}{d\lambda_d} = \sum_{\substack{\underline i\in \Z^d\\ 0\leq i_r < k}}
  \frac{\Probb(C_{\underline i})}{\lambda_d(C_{\underline i})}
  \ind{C_{\underline i}}, \; k\ge 1.
\end{equation}

By differentiation of measures we obtain $h_k \to h$, $\lambda_d$-$a.s.$ as $k\to\infty$.
Which in turn implies, owing to  Scheff\'e's Lemma,
\[
\lim_{k\to + \infty} \norm{h_k-h}_1 = 0.
\]
Furthermore, 
\[
\lim_{k\to +\infty} \normdp{h_k} = \normdp{h}
\] 
since $\normdp{h_k - h} \leq \Bigl( \lambda_d(C) \Bigr)^{\frac pd} \norm{h_k
- h}_1$ by  Jensen's Inequality applied to the probability measure
$\frac{\lambda_{d\,|C}}{\lambda_d(C)}$. Moreover, by
Proposition~\ref{prop:rateCubewise} we have
\begin{equation}\label{eq:limit_hk}
 \limn n^{1/d}\, d_{n,p}(\Probb_k) = \Qpn\, \normdp{h_k}^{\frac1p}.
\end{equation}

Likewise, we define an inner approximation of $\Probb$: 
denote by
\[
C^k = \bigcup_{C_{\underline i} \subset \opSuppP} C_{\underline i}
\]
the union of the hypercubes $C_{\underline i}$ lying  in
the interior of $\supp(\Probb)$.
Setting
\begin{eqnarray*}
\begin{split}
  \mathring \Probb_k  = \sum_{\substack{C_{\underline i} \subset \opSuppP}}
  \Probb(C_{\underline i} ) \, \U(C_{\underline i} )\quad &\mbox{ and }\quad 
  \mathring h_k &= \frac{d\mathring \Probb_k}{d\lambda_d} = h_k \ind{C^k},
\end{split}
\end{eqnarray*}
we have as above that 
\[
\mathring h_k \to h, \quad \lambda_d\mbox{-}\text{a.s.}\quad\text{ as } k\to +\infty.
\]
Consequently we also have
\[
\lim_{k\to\infty} \norm{\mathring h_k-h}_1 = 0 \quad\text{ and }\quad
\lim_{k\to\infty} \normdp{\mathring h_k} = \normdp{h}.
\]

We   get likewise by Proposition~\ref{prop:rateCubewise} that, for every
$k\in\N$,
\begin{equation}\label{eq:limit_ohk}
 \limn n^{1/d}\, d_{n,p}(\mathring \Probb_k) = \Qpn\cdot \normdp{\mathring
 h_k}^{\frac 1p}.
\end{equation}

\noindent $(a)$ Let $0 < \varepsilon < 1$ and $n \geq 2^d/\varepsilon$.
If we divide each edge of the hypercube $C$ into
\[
 m = \bigl\lfloor (\varepsilon n)^{1/d} \bigr\rfloor - 1
\]
intervals of equal length $l/m$, the interval endpoints define $m+1$ grid
points on each edge. Denoting by $\Gammaone = \Gammaone(\varepsilon, n)$ the product quantizer made up
by this procedure, we clearly have
\[
	\abs{\Gammaone} = (m+1)^d = \bigl\lfloor (\varepsilon n)^{1/d} \bigr\rfloor^d =:
	n_1.
\]
For this product quantizer it follows from  Proposition~\ref{prop:rappels} 
that, for all $\xi\in C$, 
 $$
  F^p(\xi; \Gammaone)  \leq C_{\norm{\cdot}} \sum_{i=1}^d
  \Big(\frac{l}{2m}\Big)^p  \leq C_{\norm{\cdot},p,d}\, \frac{l^p}{(\varepsilon n)^{\frac pd}}.
$$


For $n_2 = \lfloor(1-\varepsilon)n\rfloor$ let $\Gammatwo$ be an $n_2$-quantizer
such that $d^p(\Probb_k; \Gammatwo) \leq (1+\varepsilon) d^p_{n_2}(\Probb_k)$.
We clearly
have $\abs{\Gammaone \cup \Gammatwo}\leq n$ and
\begin{eqnarray*}
\begin{split}
 \hskip -0.5cm  n^{\frac pd} \biggabs{\int  F^p(\xi; \Gammaone \cup \Gammatwo) 
  d\Probb_k(\xi) - \int F^p(\xi; \Gammaone \cup \Gammatwo) d\Probb(\xi) } & \leq n^{\frac pd} \int F^p(\xi; \Gammaone \cup \Gammatwo) \abs{ h_k(\xi) - h(\xi) } d\lambda_d\xi\\ & \leq  C_{\norm{\cdot},p,d}\, \frac{l^p}{\varepsilon^{\frac pd}} \norm{h_k -
  h}_1= c_{1,\varepsilon} \norm{h_k -  h}_1 
\end{split}
\end{eqnarray*}
for $k\in\N$ and $n \geq \max\Bigl\{\frac{2^d}{\epsilon}, \frac{1}{1-\varepsilon}
\Bigr\}$. This implies
\begin{eqnarray*}
\begin{split}
  n^{\frac pd} d_n^p (\Probb) & \leq  n^{\frac pd} \int F^p(\xi; \Gammaone \cup
  \Gammatwo) d\Probb(\xi)\\ 
  &  \leq n^{\frac pd} \int F^p(\xi; \Gammaone \cup \Gammatwo) d\Probb_k(\xi) +
  c_1 \norm{h_k - h}_1\\ 
  & \leq n^{\frac pd} \int F^p(\xi; \Gammatwo) d\Probb_k(\xi) + c_1 \norm{h_k -
  h}_1\\ 
  & \leq (1+\varepsilon)\, n^{\frac pd} d_{n_2}^p(\Probb_k) + c_{1,\varepsilon}
  \norm{h_k - h}_1,
\end{split}
\end{eqnarray*}
so that we can conclude from~(\ref{eq:limit_hk}) that
\[
\limsn n^{\frac pd} d_{n}^p (\Probb) \leq
\frac{1+\varepsilon}{(1-\varepsilon)^{\frac pd}} (\Qpn)^p \normdp{h_k} +
c_{1,\varepsilon} \norm{h_k - h}_1.
\]

Letting first $k$ go to infinity and then letting $\varepsilon$ go to zero yields
\[
\limsn n^{1/d} d_n^p(\Probb) \leq \Qpn \normdp{h_k}^{\frac 1p}.
\]

\noindent $(b)$ Assume now  that $\Gammathree$ is an $n_2$-quantizer such that $\bar
d^p(\Probb;\Gammathree) \leq (1+\varepsilon) \,\bar d^p_{n_2}(\Probb)$.
%
%
Again it holds $\abs{\Gammaone \cup \Gammathree} \leq n$ and
we derive as above
\begin{equation}\label{eq:boundOpenhk}
  n^{\frac pd} \biggabs{\int  F^p(\xi; \Gammaone \cup \Gammathree) 
  d\mathring\Probb_k(\xi) - \int F^p(\xi; \Gammaone \cup \Gammathree)
  d\Probb(\xi) } \leq c_{2,\varepsilon} \norm{\mathring h_k - h}_1.
\end{equation}

Moreover, Lemma~\ref{lem:radius} yields for every $k\in\N$ the existence of
$n_{k,\varepsilon}\in\N$ such that, for all $n\geq n_{k,\varepsilon}$,
\begin{eqnarray*}
\begin{split}
 (1+\varepsilon)\, \bar d_{n_2}^p(\Probb) & \geq \bar d^p(\Probb; \Gammathree)
 \geq \int_{\conv(\Gammathree)} F^p(\xi; \Gammathree) d\Probb(\xi)\\
  & \geq \int_{C^k}
  F^p(\xi; \Gammathree) d\Probb(\xi) \geq \int_{C^k}
  F^p(\xi; \Gammaone \cup \Gammathree) d\Probb(\xi).
\end{split}
\end{eqnarray*}
Thus, we derive from~(\ref{eq:boundOpenhk}) that, for every $n\geq
\max\Big(n_{k,\varepsilon}, \frac{2^d}{\varepsilon}, \frac{1}{1-\varepsilon} \Big)$,
\begin{eqnarray*}
\begin{split}
  (1+\varepsilon)\,n^{\frac pd}\, \bar d_{n_2}^p(\Probb) & \geq n^{\frac pd}
  \int_{C^k} F^p(\xi; \Gammaone \cup \Gammathree) d\Probb(\xi)  \\
  & \geq n^{\frac pd} \int_{C^k}  F^p(\xi; \Gammaone \cup \Gammathree)
  d\mathring \Probb_k(\xi) - c_{2,\varepsilon} \norm{\mathring h_k - h}_1\\
  & \geq n^{\frac pd} d_n^p(\mathring \Probb_k) - c_{2,\varepsilon} \norm{\mathring h_k -
  h}_1,
\end{split}
\end{eqnarray*}
which yields, once  combined with~(\ref{eq:limit_ohk}),
\[
\frac{1+\varepsilon}{(1-\varepsilon)^{\frac pd}} \limin n_2^{\frac pd} \, \bar
d_{n_2,p}^p(\Probb) \geq  \Qpn \normdp{\mathring h_k}- c_{2,\varepsilon}
\norm{\mathring h_k - h}_1.
\]
Letting first $k$ go to $\infty$ and then letting $\varepsilon$ go to $0$, we get
\[
\limin n^{\frac 1d} \, \Dqbpn(\Probb) \geq  \Qpn
\normdp{h}^{\frac 1p}.\hfill \Box
\]

\begin{prop}[Singular distribution]\label{prop:rateSingular}
Assume that $\Probb$ is singular with respect to $\lambda_d$ and has compact
support. Then
 \[
  \limsn n^{\frac pd}\, \bar d_{n,p}(\Probb) = 0.
 \]

\end{prop}

\begin{proof} 
Let $A$  be a Borel set such that $\Probb(A)=1$
and $\lambda_d(A)=0$. Let $\varepsilon>0$; by the outside regularity of
$\lambda_d$, there exists an open set $O=O(\varepsilon)\supset  A$ such that
$\lambda_d(O)\le \varepsilon$ (and $\Probb(O)=1$). Let $C$ be an open hypercube with
edges parallel to the coordinate axis, edge-length $\ell$ and  containing the closure of
$A$.

Let $C_k= \prod_{i=1}^d [c_{k,i}, c_{k,i}+\ell_i)$, $k\!\in \N$, be a countable
partition of $O$ consisting of nonempty half-open hypercubes, still with edges
parallel to the coordinate axis (see,  $e.g.$ Lemma 1.4.2 in~\cite{COH}).

Let $m=m(\varepsilon)\!\in \N$ such that $\displaystyle \sum_{k\ge m+1}\Probb(C_k)\le
\varepsilon^{\frac pd} \ell^{-p}$.

Let $n\!\in \N$, $n\ge 2^{d+1}$ and let $n_1, \ldots, n_d\ge 2$ be integers such
that the product $n^d_1 + \cdots + n_m^d \le n/2$. One designs a grid $\Gamma$
as follows.

For every $k\!\in \{1,\ldots,m\}$, we consider the lattice of $C_k$ of size $n_i
^d$ defined by 
\[ 
\Prod_{i=1}^d \Bigl\{c_{k,i}+ \frac{r_i}{n_k-1}\ell_i,\,
r_i=0,\ldots, n_k-1,\, i=1,\ldots, d\Bigr\}. 
\]

Then, one defines likewise the lattice of $C$ of size $n_{m+1}^d \le n/2$ 
\[
\Prod_{i=1}^d \Bigl\{c_{k,i}+ \frac{r_i}{n_{m+1}-1}\ell_i,\, r_i=0,\ldots,
n_{m+1}-1,\, i=1,\ldots,d\Bigr\}. 
\] 
The grid $\Gamma$ is made up with all the points
of the $m+1$ above finite lattices.

Now let $\xi \!\in A$. It is clear from the definition of the function $F_p$
that 
\[ F_p(\xi; \Gamma) \le \left\{\begin{array}{ll}
C_{\norm{.}} \big(\ell_k/n_k\big)^p& \mbox{if } \; \xi \!\in \bigcup_{k=1}^mC_k \\
C_{\norm{.}} \big(\ell/n_{m+1}\big)^p& \mbox{if } \; \xi \!\in C \setminus 
\bigcup_{k=1}^mC_k        \end{array}\right. 
\] where $C_{\norm{.}} >0$ is a
real constant only depending on the norm. As a consequence
\begin{eqnarray*}
d_{n}^p(\Probb)&=& \sum_{k=1}^m \int_{C_k}F^p(\xi;\Gamma)d\Probb(\xi) +\int_{C \setminus  \bigcup_{k=1}^mC_k } F^p(\xi; \Gamma)d\Probb(\xi)\\
&\le& C_{\norm{.}}  \Big( \sum_{k=1}^m (\ell_k/n_k)^p\Probb(C_k) + (\ell/n_{m+1})^p\Probb (C \setminus  \bigcup_{k=1}^mC_k )\Big).
\end{eqnarray*}

Set for every $k \!\in \{1,\ldots,m\}$, $\displaystyle n_k =\left \lfloor \frac{\ell_k(n/2)^{\frac 1d}}{(\sum_{k'=1}^d\ell_{k'}^d)^{\frac 1d}}\right \rfloor$ 
and $\displaystyle n_{m+1} = \lfloor (n/2)^{\frac 1d} \rfloor$. Note that 
\[
\sum_{k'=1}^d\ell_{k'}^d= \sum_{k=1}^m \lambda_d(C_k) \le \lambda_d(O)\le \varepsilon.
\]
Elementary computations show that for large enough $n$,  all the integers $n_k$ are greater
than $1$ and that 
\begin{eqnarray*} 
\sum_{k=1}^m (\ell_k/n_k)^p\Probb(C_k) + (\ell/n_{m+1})^p\Probb(C \setminus  \bigcup_{k=1}^mC_k ) & \le & (\sum_{k'=1}^d\ell_{k'}^d)^{\frac pd}(n/2)^{-\frac pd} \Probb\big(\cup_{1\le k\le m} C_k\big) + \\
&& +(n/2)^{-\frac pd} \ell^p \Probb\big(C \setminus  \bigcup_{k=1}^mC_k \big)\\ 
\end{eqnarray*}
so that \[ \limsup_n n^{\frac pd} d_{n}^p(\Probb) \le C_{\norm{.}} (\varepsilon/2)^{\frac pd}\]
which in turn
implies, by letting $\varepsilon $ go to $0$, that $\displaystyle \limsup_n n^{\frac pd}
d_{n}^p(\Probb)=0$. 
\end{proof}

 {\sc Proof of Theorem~\ref{thm:DQRate}:}
 Claim~(a) follows directly from Propositions~\ref{prop:rateCompact} , 
~\ref{prop:rateSingular} and  Proposition~\ref{prop:subLin}: Assume $\Probb = \rho \Probb_a +(1-\rho)\Probb_s$ where $\Probb_a= \frac {h}{\rho}\lambda_d$ and $\Probb_s$ denote the absolutely continuous and singular part of $\Probb$ respectively. The following inequalities hold true 
\[
 \rho \bar d_{n,p} (\Probb_a)\le \bar d_{n,p}(\Probb)\le \rho \bar d_{n_1,p}(\Probb_a)+(1-\rho) \bar d_{n_2,p}(\Probb_s)
 \]
 for every triplet of integers $(n_1,n_2,n)$ with $n_1+n_2\le n$. Set $n_1= n_1(n)= \lfloor (1- \varepsilon n) \rfloor$, $n_2= n_2(n)= \lfloor  \varepsilon n \rfloor$. Then we derive that 
 \[
\hskip -1cm\rho  \Qpn\cdot \normdp{\frac{h}{\rho} }^{\frac 1p} \liminf_n n^{\frac p d}\bar d_{n,p} (\Probb_a)\le  \liminf_n n^{\frac p d}\bar d_{n,p} (\Probb)\le  \limsup_n n^{\frac p d}\bar d_{n,p} (\Probb)\le \rho (1-\varepsilon)^{-\frac p d}  \Qpn\cdot \normdp{\frac{h}{\rho}}^{\frac 1p}
 \]
 Letting $\varepsilon $ go to $0$ completes the proof.
 
 Furthermore, part $(c)$ 
 was derived in~\cite{dualStat}, Section~5.1. 
 Hence, it remains to prove $(b)$ 
  \begin{proof} {\sc Step 1.} (Lower bound)
If $X$ is compactly supported, the assertion follows from Proposition
\ref{prop:rateCompact}. Otherwise, set for every $R\!\in(0,\infty)$, 
$$
C_{_R} = [-R,R]^d\; \mbox{ and }\; \Probb(\cdot | C_k) = \frac{h\ind{C_k}}{\Probb(C_k)} \lambda_d,\; k\in\N.
$$
Proposition~\ref{prop:rateCompact} yields again
\begin{equation}\label{eq:limitCk}
  \limn n^{\frac 1d} \, \bar d_{n,p}(\Probb(\cdot | C_k)) =  \Qpn \cdot
\normdp{h\ind{C_k}/\Probb(C_k)}^{\frac 1p},
\end{equation}
so that $\dqbpn(\Probb) \geq \Probb(\cdot | C_k) \dqbpn(\Probb(\cdot |
C_k))$ implies for all $k\in\N$
\[
\limin n^{\frac 1d} \, \bar d_{n,p}(\Probb) \geq  \Qpn \cdot
\normdp{h\ind{C_k}}^{\frac 1p}.
\]
Sending $k$ to infinity, we get at
\[
\limin n^{\frac 1d} \, \bar d_{n,p}(\Probb) \geq  \Qpn \cdot
\normdp{h}^{\frac 1p}.
\]

\noindent  {\sc Step 2} ({\em Upper bound}, ${\rm supp}(\Probb)= \R^d$). Let
$\rho\!\in(0,1)$. Set $K=C_{k+\rho}$ and $K_{\rho}= C_k$. Let
$\Gamma_{k,\alpha,\rho}$ be the lattice grid associated to $K\setminus K_{\rho}$
with edge $\alpha>0$ as defined  in the proof of Proposition~\ref{prop:rateCubewise}.
It is straightforward that there exists a real constant $C>0$ such that
\[
\forall\, k\in\N, \forall\, \rho\in(0,1), \forall\,\alpha\in(0,\rho):\quad
|\Gamma_{\alpha,\rho}|\le C d\rho k^{d-1}\alpha^{-d}  .
\] 
Let $\varepsilon\!\in (0,1)$. For every $n\ge 1$, set $\alpha_n = \tilde
\alpha_0 n^{-\frac 1d}$ where $\tilde \alpha_0\in(0,1)$ is a real constant and
\[ 
n_0= |\Gamma_{k,\alpha_n,\rho}|,\quad n_1=\lfloor
(1-\varepsilon)(n-n_0)\rfloor,\quad n_2= \lfloor \varepsilon (n-n_0)\rfloor, 
\]
so that $\alpha_n\in(0.\rho)$, $n_0+n_1+n_2\le n$ and $n_i\ge1$ for large enough
$n$. 

\smallskip For every $\xi\!\in K_{\rho}=C_k$, for every grid $\Gamma\subset
\R^d$ containing $K_{\rho}$, we know by the ``firewall" 
Lemma~\ref{lem:firewall} that
\[ 
F^p(\xi;(\Gamma\cap
\mathring{K})\cup\Gamma_{\alpha,\rho})\le (1+\eta)^p F^p(\xi;\Gamma)
+(1+\eta)^p(1+1/\eta)^pC_{\norm{.}}\alpha^p. 
\]

Let $\Gammaone=\Gammaone(n_1,k)$ be an $n_1$ quantizer such that
$d^p_{n_1}(\Probb(.|C_k); \Gammaone) \leq (1+\eta) d^p_{n_1}(\Probb(.|C_k))$.
Set $\Gammaone'=((\Gammaone\cap \mathring C_{k+\rho})\cup
\Gamma_{k,\alpha_n,\rho})$. One has $\Gammaone'\subset C_{k+2\rho}$ for large
enough $n$ (so that $\alpha_n<\rho$).

Let moreover $\Gammatwo=\Gammatwo(n_2,k)$ be an $n_2$ quantizer such that
$\bar d^p_{n_2}(\Probb(.|C_k^c); \Gammatwo) \leq (1+\eta)
\bar d^p_{n_2}(\Probb(.|C_k^c))$. 
For $n\ge n_{\rho}$, we may assume that
$C_{k+2\rho}\subset \conv{\Gammatwo}$ owing to Lemma~\ref{lem:radius} since
$C_{k+2\rho} =\conv(C_{k+2\rho}\setminus C_{k+\frac 32\rho})$ and
$C_{k+2\rho}\setminus C_{k+\frac 32\rho}\subset \mathring{\overbrace{{\rm supp}
\Probb(.|C_k^c)}}$.
As a consequence $\Gammaone'\subset \conv(\Gammatwo)$ so that
$\conv(\Gammaone')\subset \conv (\Gammatwo)=\conv(\Gamma)$ where
$\Gamma=\Gammaone'\cup \Gammatwo$and \[ C_{k+\rho}\subset
\conv(\Gamma)=\conv(\Gammatwo). 
\] 
Now
\begin{eqnarray*}
\bar d_{n}^p(\Probb) & \le & \int_{C_k}\Big(F^p(\xi;\Gamma)\mbox{\bf
1}_{\{\xi \in \conv (\Gammatwo)\}}+ \underbrace{ d(\xi,\Gamma)^p\mbox{\bf
1}_{\{\xi \notin \conv (\Gammatwo)\}} }_{=0}\Big) d\Probb(\xi)\\ &&+
\int_{C^c_k}\left(F^p(\xi;\Gamma)\mbox{\bf 1}_{\{\xi \in \conv (\Gammatwo)\}}+
d(\xi,\Gamma)^p \mbox{\bf 1}_{\{\xi \notin \conv
(\Gammatwo)\}}\right)d\Probb(\xi).
\end{eqnarray*}
Using that, for every $\xi\!\in C_k$, 
\begin{eqnarray*}
F^p(\xi;\Gamma) &\le& F^p(\xi;\Gammaone')\\
&\le& (1+\eta)^p\Big(F^p(\xi;\Gammaone)+
(1+1/\eta)^p\,C_{\norm{.}}\,\alpha_n^p\Big)
\end{eqnarray*}
implies
\begin{eqnarray*}
\bar d_{n}^p(\Probb) & \le & \Probb(C_k) (1+\eta)^p\Big(
(1+\eta)\,d^p_{n_1}(\Probb(.|C_k)) + (1+1/\eta)^p\,C_{\norm{.}}\,\tilde
\alpha_0 \,n^{-\frac 1d} \Big)\\
 && + \Probb(C^c_k)\,(1+\eta)\, \bar d^p_{n_2}(\Probb(.|C^c_k)).
 \end{eqnarray*}
 Consequently
 \begin{eqnarray*}
n^{\frac pd} \bar d_{n}^p(\Probb) & \le & \Probb(C_k)
(1+\eta)^p\Big[(1+\eta)\, \Big(\frac{n}{n_1}\Big)^{\frac pd}\,n_1^{\frac
pd}\,d^p_{n_1}(\Probb(.|C_k)) 
+ (1+1/\eta)^pC_{\norm{.}}\tilde \alpha_0 \Big]\\ 
 && + (1+\eta)\,
\Big(\frac{n}{n_2}\Big)^{\frac pd}\Probb(C^c_k)\, n_2^{\frac pd}\,\bar
d_{n_2}^p(\Probb(.|C^c_k))
 \end{eqnarray*}
 which in turn implies, using Proposition~\ref{prop:rateCompact} for the modulus $d_{n,p}$ and the $d$-dimensional version of the extended Pierce
 Lemma (Proposition~\ref{PdtQErrop}) for $\bar d_{n,p}$,
 \begin{eqnarray*}
 \limsup_nn^{\frac pd} \bar d_{n}^p(\Probb)  &\le&  \Probb(C_k)
 (1+\eta)^p\left(\left(\frac{(1+\eta)^{-p/d}}{(1-\varepsilon)(1-Cd\rho
 k^{d-1}\tilde \alpha^{-d}_0)}\right)^{\frac pd}Q^{dq}_{\norm{.}}\norm{h\mbox{\bf 1}_{C_k}}_{L^{\frac{d}{d+p}}}\right.\\ &&\left. 
 +  (1+1/\eta)^pC_{\norm{.}}\tilde \alpha_0 \right)\\ 
 &&+ \Probb(C^c_k) \,(1+\eta)\,\,C_{p,d}\,\norm{X \mbox{\bf 1}_{\{X\in C^c_k\}}
 }^p_{L^{p+\delta}} \left(\frac{1}{\varepsilon(1-Cd\rho k^{d-1}\tilde \alpha^{-d}_0)}\right)^{\frac pd}.
 \end{eqnarray*}
 One concludes by letting successively $\rho$, $\tilde \alpha_0$, $\eta$ go
 to $0$,   $k \to \infty$ and finally $\varepsilon$ to  $0$.

\medskip
\noindent {\sc Step 3.} (Upper bound: general case). Let $\rho\!\in (0,1)$. Set
$\Probb_{\rho}= \rho \Probb+(1-\rho)\Probb_0$ where $\Probb_0= {\cal N}(0;I_d)$
($d$-dimensional normal distribution). It is clear from the very definition of
$\bar d_{n,p}$ that $\bar d_{n,p}(\Probb)\le \frac{1}{\rho} \bar
d_{n,p}(\Probb_{\rho})$ since $\Probb\le \frac{1}{\rho} \Probb_{\rho}$. The
distribution $\Probb_{\rho}$ has $h_{\rho}= \rho h+(1-\rho)h_0$ as a density (with obvious
notations) and one concludes by noting that 
\[ 
\lim_{\rho\to
0}\|h_{\rho}\|_{d/(d+p)}= \|h\|_{d/(d+p)} 
\] 
owing to the Lebesgue dominated convergence Theorem.
\end{proof}

{\em Proof of Proposition~\ref{prop:asymptQ}:}
Using H\"older's inequality one easily checks that for $0\leq r\leq p$ and
$x\in\R^d$ it holds
\[
	\abs{x}_{\ell^r} \leq d^{\frac{1}{r} - \frac{1}{p}}\, \abs{x}_{\ell^p}.
\]
 
Moreover, for $m\in\N$ set $n = m^d$ and let $\Gamma'$ be an optimal quantizer
for $d_{m,p}(\Unif)$ (or at least $(1+\varepsilon)$-optimal for $\varepsilon > 0$).
Denoting $\Gamma = \prod_{i=1}^d \Gamma'$, it then follows from Proposition~\ref{prop:rappels}$(b)$ that
\[
n^\frac{p}{d}\, d^p_n\big(\Unifd\big) \leq n^\frac{p}{d}\, d^p\big(\Unifd;\Gamma\big) = m^p
	\sum_{i=1}^d d^p\big(\Unif; \Gamma'\big)  = d\, m^p \,d^p_m\big(\Unif\big).
\]

Combining both results and reminding that $\Qpn$ holds as an infimum, we obtain
for $r\!\in [0,p]$,  
\[
\bigl(Q^{\text{dq}}_{\abs{\cdot}_{\ell^r}, p, d}\bigr)^p \leq d^{\frac{p}{r}-1}
\, n^{\frac{p}{d}}\, d^p_{n, \abs{\cdot}_{\ell^p}}(\Unifd) \leq d^{\frac{p}{r}}
\, m^p \,d^p_m\big(\Unif\big),
\]
which finally proves the assertion by sending $m\to+\infty$.$\qquad \Box$

\medskip
\section{Concluding  remarks and prospects} This result does not complete the theoretical investigations about dual quantization (beyond the existence of optimal dual quantizers in the case $p=1$, left open in~\cite{dualStat}): the first one is to elucidate the asymptotic behaviour of the constant $Q^{dq}_{\norm{.},p,d}$ coming out in Theorem~\ref{thm:DQRate} as $d$ goes to infinity,  most likely by showing that  $\lim_{d\to +\infty} \frac{Q^{dq}_{\norm{.},p,d}}{Q^{vq}_{\norm{.},p,d}}=1$. From a practical point of view, is it possible to evaluate the mean dual quantization error induced by an optimal Voronoi quantization grid? An answer to that question would be very valuable for applications since many optimal quantization grids have been computed for various distributions (see $e.g.$~\cite{Website} for Gaussian distributions).

Many natural questions  solved  in the optimal Voronoi quantization theory remain open. Among others  ``Is there  a counterpart to the empirical measure theorem for (asymptotically) optimal quantizers?" (see~Theorem~7.5, p.96~in~\cite{Foundations})? ``How does dual  quantization  behave with respect to empirical distribution of i.i.d. $n$-samples of a given distribution?".  Is it possible to develop an infinite dimensional ``functional" dual quantization?

\appendix
\section{Numerical results for $\bar d_{n,2}(X)^2$}\label{app:num}
In order to support the heuristic argumentation on the intrinsic and rate optimal growth limitation of the truncation error $\Probb\big(X\notin C_{n}\big)$ induced by the extended dual quantization error modulus,
we consider the two dimensional random variable 
\[
	X = (W_T, \sup_{0\leq t \leq T}W_t),
\]
where $(W_t)_{0\leq t \leq T}$ is a Brownian Motion.

This example is motivated by the pricing of exotic options, where this joint distribution plays an important role.

Using a variant of the CVLQ algorithm (see \cite{dualStat}) adapted for the dual quantization modulus inside $C_n$ and the nearest neighbor mapping outside, we have computed a sequence of optimal grids together with the squared dual quantization error $\bar d_{n,2}(X)^2$ and the truncation error $\Probb\big(X\notin C_{n}\big)$.

These results are reported in Table~\ref{tab:num}. 

\begin{table}[h!]
\centering
\label{tab:num}
\caption{Numerical results for the dual quantization $X$}
\begin{tabular}{l|c|c}
$n$ & $\bar d_{n,2}(X)^2$ & $\Probb\big(X\notin C_{n}\big)$ \\
\hline
50 & 0.04076 & 0.01784\\
100 & 0.01966 & 0.00795 \\
150 & 0.01236 & 0.00412 \\
200 & 0.00931 & 0.00141 \\
\end{tabular}
\end{table}

Furthermore we see in figure \ref{fig:num} a log-log plot for the convergence of the two rates $\bar d_{n,2}(X)^2$ and $\Probb\big(X\notin C_{n}\big)$.

\begin{figure}[h!]
\includegraphics{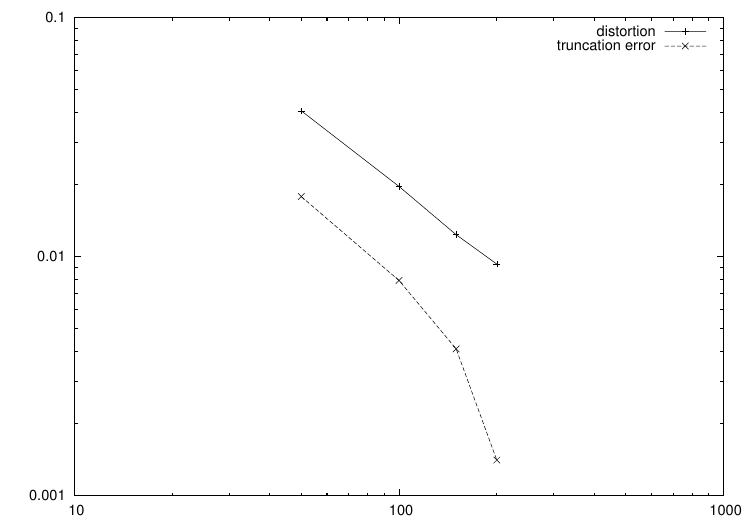}
\caption{log-log plot of $\bar d_{n,2}(X)^2$ and $\Probb\big(X\notin C_{n}\big)$ with respect to the grid size $n$}
\label{fig:num}
\end{figure}

The distortion rate $\bar d_{n,2}(X)^2$ shows here an absolute stable convergence rate (least-squares fit of exponent yields $-1.07192$) which is consistent with the theoretical optimal rate of $n^{-\frac{2}{d}}$.
Moreover, the truncation error $\Probb\big(X\notin C_{n}\big)$ outperforms also in this case the heuristically derived rate of $n^{-1}$ and also outperforms the squared "inside" quantization error, which means that also for such an un-symmetric and non-spherical distribution of the Brownian motion and its supremum, an second order rate can be achieved.

This confirms again the motivation of the extended dual quantization error as the correction penalization constraint on growth of the convex hull in order to preserve second order stationarity.

\small
\bibliography{myliterature}

\end{document}